\documentclass{amsart}

\usepackage[dvipsnames]{xcolor}
\usepackage[margin=1in]{geometry} 
\usepackage{amsmath, amsfonts, amsthm, amstext, amssymb, mathtools, thmtools, bbm} 
\usepackage[colorlinks=true,linkcolor=blue,citecolor=blue]{hyperref} 
\usepackage{todonotes} 
\usepackage[nameinlink, capitalize,noabbrev]{cleveref} 
\usepackage{tikz-cd, tikz} 
\usepackage[shortlabels]{enumitem} 
\usepackage[mathscr]{euscript} 
\usepackage{stmaryrd} 
\usepackage{adjustbox}
\usepackage{nicefrac}

\usepackage{wasysym}

\usetikzlibrary{shapes.geometric, arrows.meta, calc, positioning, svg.path, arrows}

\crefname{equation}{Equation}{}

\numberwithin{equation}{section} 
\numberwithin{table}{section} 
\numberwithin{figure}{section}

\theoremstyle{definition}
\newtheorem{definition}[equation]{Definition}

\newtheorem{remark}[equation]{Remark}
\newtheorem{example}[equation]{Example}
\newtheorem{notation}[equation]{Notation}

\theoremstyle{plain}
\newtheorem{lemma}[equation]{Lemma}
\newtheorem{proposition}[equation]{Proposition}
\newtheorem{prop}[equation]{Proposition}
\newtheorem{theorem}[equation]{Theorem}
\newtheorem{corollary}[equation]{Corollary}

\newtheorem*{warning}{Warning}
\newtheorem*{heuristic}{Heuristic}

\newtheorem{maintheorem}{Theorem} 
\newtheorem*{maincorollary}{Corollary}
 
\crefname{maintheorem}{Theorem}{Theorems}

\newcommand{\ul}[1]{\underline{#1}} 

\newcommand{\K}{\mathcal{K}} 
\newcommand{\Z}{\mathbb{Z}} 
\newcommand{\1}{\mathbbm{1}} 
\newcommand{\scrO}{\mathscr{O}} 
\newcommand{\C}{\mathscr{C}} 
\newcommand{\D}{\mathscr{D}} 
\renewcommand{\phi}{\varphi} 
\renewcommand{\epsilon}{\varepsilon} 

\newcommand{\upi}{\ul{\pi}}
\newcommand{\uA}{\underline{A}}

\newcommand{\sphere}{\mathbb{S}} 

\DeclareMathOperator{\coker}{coker} 
\DeclareMathOperator{\res}{res} 
\DeclareMathOperator{\tr}{tr} 
\DeclareMathOperator{\Lan}{Lan} 
\DeclareMathOperator{\RO}{RO} 
\DeclareMathOperator{\Pic}{Pic} 
\DeclareMathOperator{\Ho}{Ho} 
\DeclareMathOperator{\ev}{ev} 
\DeclareMathOperator{\qRes}{qRes} 
\DeclareMathOperator{\qInd}{qInd} 
\DeclareMathOperator{\Inf}{Inf} 
\DeclareMathOperator{\Hom}{Hom} 
\DeclareMathOperator{\DressName}{Dress} 
\DeclareMathOperator{\Spec}{Spec} 
\DeclareMathOperator{\im}{im} 
\DeclareMathOperator{\id}{id} 
\DeclareMathOperator{\End}{End} 
\DeclareMathOperator{\Sub}{Sub} 
\DeclareMathOperator{\Iso}{Iso} 
\DeclareMathOperator{\Fix}{Fix}
\DeclareMathOperator{\Dim}{Dim}

\newcommand{\Ab}{{\mathscr{A}\textnormal{b}}} 
\newcommand{\CRing}{{\textnormal{C}\mathscr{R}\textnormal{ing}}} 
\newcommand{\Mod}{{\mathscr{M}\textnormal{od}}} 
\newcommand{\Vect}{{\mathscr{V}\textnormal{ect}}} 
\newcommand{\Mack}{{\mathscr{M}\textnormal{ack}}} 
\newcommand{\Green}{{\mathscr{G}\textnormal{reen}}} 
\newcommand{\Set}{{\mathscr{S}\textnormal{et}}} 
\newcommand{\Sp}{\mathscr{S}\textnormal{p}} 
\newcommand{\CMon}{\textnormal{C}\mathcal{M}\textnormal{on}} 
\newcommand{\SymMonCat}{\mathscr{S}\textnormal{ym}\mathscr{M}\textnormal{on}\mathscr{C}\textnormal{at}} 
\newcommand{\B}{\mathscr{B}} 

\newcommand{\mylittlematrix}[1]
{\scalebox{1}{$\begin{pmatrix} #1 \end{pmatrix}$}}
\newcommand{\mymidsizematrix}[1]
{\scalebox{0.7}{$\begin{pmatrix} #1 \end{pmatrix}$}}
\newcommand{\mytinymatrix}[1]
{\scalebox{0.6}{$\begin{pmatrix} #1 \end{pmatrix}$}}

\newcommand{\adjunction}[4]{
\begin{tikzpicture}
\node (00) at (0,0) {#1};
\node (20) at (4,0) {#2};
\node at (2,0) {$\perp$};

\draw[->, out=15, in=165] (00) to node[above] {#3} (20);
\draw[->, out=195, in=-15] (20) to node[below] {#4} (00);
\end{tikzpicture}
}

\newcommand{\SigmaTwoMack}[4]{
\begin{tikzpicture}
\node (02) at (0,2) {#1};
\node (00) at (0,0) {#3};

\draw[->, bend right] (02) to node[left] {#2} (00);
\draw[->, bend right] (00) to node[right] {#4} (02);
\end{tikzpicture}
}

\newcommand{\tw}[1]{T_{#1}}
\newcommand{\twinv}[1]{T_{#1}^{\mathrm{inv}}}
\newcommand{\tres}[3]{\res^{#1}_{#2,#3}}
\newcommand{\Dress}[3]{\DressName(#1,#2;#3)}

\newcommand{\sfrac}[2]{{}^{#1}\!/_{\!#2}} 

\usepackage[shortalphabetic]{amsrefs}

\begin{document}

\title{Picard Groups in Equivariant Algebra and Stable Homotopy Theory}

\begin{abstract}
Traditionally, homotopy groups in $G$-equivariant stable homotopy theory have been graded over $\RO(G)$, the real representation ring of $G$. It is arguably more natural to grade homotopical structures over the Picard group of the equivariant stable homotopy category. Though there is a canonical map of abelian groups $\RO(G) \rightarrow \Pic(\Ho(\Sp^G))$ relating the two, this map is neither injective or surjective in general. Fausk, Lewis, and May give an algebraic expression of $\Pic(\Ho(\Sp^G))$ in terms of the Picard group of the Burnside ring $A(G)$, and this work suggests a folklore isomorphism between $\Pic(A(G))$ and $\Pic(\Mack_G)$. We prove the existence of this folklore isomorphism in the setting of finite groups, then leverage our analysis to prove a classification of invertible Mackey functors in the setting of finite abelian groups. As a consequence, we furnish a classification of invertible $A(G)$-modules again for $G$ a finite abelian group. 
\end{abstract}

\author{Jesse Keyes and Jordan Sawdy}

\maketitle
\tableofcontents

\section{Introduction}
Constructing the $G$-equivariant stable homotopy category, $\Ho(\Sp^G)$, involves inverting the real representation spheres associated to $G$ with respect to the smash product of $G$-spaces. Consequently, we can consider the homotopy of any object in $\Ho(\Sp^G)$ as $\RO(G)$-graded. Historically, the structure of $RO(G)$-graded homotopy groups has been used in homological settings when a theory of Poincar\'e duality is desired.
It is also known that $RO(G)$-graded homotopy groups detect weak equivalences for $G=C_2$, so often the choice to compute these groups is made to forego computing homotopy Mackey functors. On the other hand, $\Z$-graded homotopy Mackey functors detect weak equivalences by definition. 
There is yet another group over which one can grade homotopy groups or Mackey functors: the Picard group of $\Ho(\Sp^G)$. Grading homotopy groups or Mackey functors over the Picard group of $\Ho(\Sp^G)$ is justified formally in that no choice is made in what invertible objects (up to isomorphism) we choose to grade homotopical structures over. {Moreover, the Picard group of $G$-spectra is often smaller than the real representation ring of $G$ (see \cite[Example 1.1]{Angeltveit}).}

Essentially, two choices are being made. The first choice is whether to regard $\Ho(\Sp^G)$ as enriched in abelian groups or in $G$-Mackey functors, though these are related via change of enrichment along the lax monoidal functor $\ev_{G/G}\colon \Mack_G \rightarrow \Ab$. The second choice comes from a more general formalism discussed throughout \cite{DuggerGrading}: given any additive (or $\Mack_G$-enriched) closed symmetric monoidal category $(\mathscr{C}, \otimes, \mathbbm{1})$, one can define a multi-graded abelian group (respectively, $G$-Mackey functor) $\pi_*(X)$ for any $X \in \mathscr{C}$. Importantly, the grading here can be over any subgroup $A \leq \Pic(\mathscr{C})$, where $\Pic(\mathscr{C})$ denotes the Picard group of $\mathscr{C}$. From this, it is clear that grading over $\Pic(\mathscr{C})$ is a canonical choice. 

{In \cite[Theorem 0.1]{FLM}, Fausk, Lewis, and May give an algebraic description of $\Pic(\Ho(\Sp^G))$ as an extension of some free abelian group by $\Pic(A(G))$, where $A(G)$ is the Burnside ring for $G$.} In particular, they prove that there is an embedding 
\[ 
\Pic(A(G)) \coloneqq \Pic(\Mod_{A(G)}) \cong \Pic(\Mod_{\pi_0(\mathbb{S}_G)}) \hookrightarrow \Pic(\Ho(\Sp^G))
\] 
whose cokernel is given by some free abelian group. Further, $\Pic(A(G))$ is computed by tom Dieck and Petrie; see the discussion following \cite[Proposition 11.3]{HmtpyReps}. If one regards $\Ho(\Sp^G)$ as enriched in $\Mack_G$, the work of \cite{FLM} raises the question of whether an embedding 
\[
\Pic(\Mack_G) \cong \Pic(\Mod_{\underline{A}}) \cong \Pic(\Mod_{\underline{\pi}_0(\mathbb{S}_G)}) \hookrightarrow \Pic(\Ho(\Sp^G))
\]
with the same free cokernel exists. Such an embedding would imply an isomorphism between $\Pic(\Mack_G)$ and $\Pic(A(G))$. It is folklore that there exists an isomorphism $\Pic(\Mack_G) \rightarrow \Pic(A(G))$ induced by evaluating invertible Mackey functors at level $\sfrac{G}{G}$. 

\subsection{Outline of Results}

The first goal of this article is to give a precise account of the folklore statement above: for $G$ a finite group, we build an isomorphism $\Pic(A(G)) \to \Pic(\Mack_G)$. This analysis goes through $\Pic(\Ho(\Sp^G))$ along the same lines as \cite{FLM}. We define a full symmetric monoidal subcategory, $k\Mack_G$, of $\Mack_G$ consisting of $G$-Mackey functors which satisfy a finiteness condition analogous to that in \cite{FLM} for K\"unneth objects.        

\begin{maintheorem}
{[\cref{Theorem A}]}
There is a split exact sequence 
\[ 0 \rightarrow \Pic(k\Mack_G) \rightarrow \Pic(\Ho(\Sp^G)) \xrightarrow{\Dim} \Z^r \rightarrow 0 \] 
in which the image of $\Pic(k\Mack_G)$ in $\Pic(\Ho(\Sp^G))$ is precisely the isomorphism classes of invertible K\"unneth objects in $\Ho(\Sp^G)$. In particular, we learn that $\Pic(A(G))$ is isomorphic to $\Pic(k\Mack_G)$.   
\end{maintheorem}

To obtain an isomorphism $\Pic(k\Mack_G) \cong \Pic(\Mack_G)$, we show that the kernel of the dimension function above contains $\Pic(\Mack_G)$. 

\begin{maintheorem}
{[\cref{Theorem B}]}
There exists an embedding $\Pic(\Mack_G) \hookrightarrow \Pic(\Ho(\Sp^G))$ whose image is contained in the kernel of the dimension function. Consequently, the inclusion $k\Mack_G \subseteq \Mack_G$ induces an isomorphism on Picard groups. 
\end{maintheorem}

Following this identification of $\Pic(A(G))$ with $\Pic(\Mack_G)$, we build explicit representatives for the isomorphism classes of invertible Mackey functors, and our approach is largely inspired by that of \cite[Definition 4.5]{Angeltveit}. This classification furnishes that of invertible $A(G)$-modules, recovering the computation of $\Pic(A(G))$ given by \cite{HmtpyReps}. 

\begin{maintheorem}
{[\cref{Theorem C}, \cref{ClassificationThm}] For $G$ a finite abelian group,} there is an isomorphism 
\[ \twinv{G} \coloneqq \prod_{H \leq G} \left( \sfrac{\Z}{[G:H]} \right)^\times / \{\pm 1\} \rightarrow \Pic(\Mack_G) \]
given by sending $\alpha \in \twinv{G}$ to the corresponding twisted Burnside Mackey functor. That is, every invertible Mackey functor is isomorphic to a twisted Burnside Mackey functor.    
\end{maintheorem}

The main technical hurdle in justifying the folklore isomorphism $\Pic(\Mack_G) \rightarrow \Pic(A(G))$ is that the functor $\ev_{\sfrac{G}{G}}\colon \Mack_G \to \Mod_{A(G)}$ is not strong symmetric monoidal, so it need not send invertible objects to invertible objects in general. However, this can be rectified by restricting the domain, as shown in the following result: 

\begin{maintheorem}
{[\cref{EvalEquiv}]}
The functor $\ev_{\sfrac{G}{G}}\colon k\Mack_G \rightarrow k\Mod_{A(G)}$ given by evaluation at $\sfrac{G}{G}$ is a strong symmetric monoidal equivalence of categories. That is, it induces an isomorphism 
\[ \Pic(k\Mack_G) \cong \Pic(k\Mod_{A(G)}) \cong \Pic(A(G)). \]
We note that K\"unneth objects in $\Mod_{A(G)}$ coincide with dualizable objects (i.e., finitely generated projective modules). 
\end{maintheorem} 

\begin{maincorollary}
{[\cref{maincorollary}]}
The map 
\[ \twinv{G} \rightarrow \Pic(k\Mack_G) \xrightarrow{\ev_{\sfrac{G}{G}}} \Pic(k\Mod_{A(G)}) \cong \Pic(A(G)) \]
sending $\alpha \in \twinv{G}$ to the corresponding twisted Burnside module is an isomorphism. That is, every invertible $A(G)$-module is isomorphic to a twisted Burnside module.
\end{maincorollary}

\subsection{Notation}
Throughout, $G$ will denote a finite group, unless specified otherwise. Let $(\C,\otimes,\mathbbm{1})$ be a symmetric monoidal category. 
\begin{itemize}
\item Categories.
	\begin{itemize}[-]
		\item $\CRing$: commutative rings
		\item $\Mod_R$: modules over a commutative monoid in some symmetric monoidal category
		\item $\Mack_G$: $G$-Mackey functors
		\item $\Green_G$: $G$-Green functors
		\item $\Ho(\Sp)$: the stable homotopy category
		\item $\Ho(\Sp^G)$: the $G$-equivariant stable homotopy category
		\item $\SymMonCat$: symmetric monoidal categories and strong symmetric monoidal functors
		\item $d\C$: full subcategory of dualizable objects in $(\C, \otimes, \mathbbm{1})$
		\item $k\C$: full subcategory of K{\"u}nneth objects in $(\C, \otimes, \mathbbm{1})$
		\item $\CMon(\C)$: category of commutative monoids in $(\C,\otimes,\mathbbm{1})$
	\end{itemize}
\item Functors. Let $J \leq H$ be finite groups and $N \lhd H$ with $q\colon G \twoheadrightarrow \sfrac{G}{N}$. 
	\begin{itemize}[-]
		\item $\uparrow_J^H\colon \Mack_H \rightarrow \Mack_J$: induction 
		\item $\downarrow_J^H\colon \Mack_J \rightarrow \Mack_H$: restriction 
		\item $\qInd_{\sfrac{H}{N}}^H\colon \Mack_{\sfrac{H}{N}} \rightarrow \Mack_H$: induction along a quotient map
		\item $\qRes_{\sfrac{H}{N}}^H\colon \Mack_{H} \rightarrow \Mack_{\sfrac{H}{N}}$: restriction along a quotient map
		\item $\Inf_{\sfrac{H}{N}}^H\colon \Mack_{\sfrac{H}{N}} \rightarrow \Mack_H$: inflation 
		\item $\Phi^N\colon \Mack_H \rightarrow \Mack_{\sfrac{H}{N}}$: $N$-geometric fixed points
		\item $\Phi^H\colon \Ho(\Sp^G)) \rightarrow \Ho(\Sp)$: $H$-geometric fixed points 
		\item $\Pic\colon \SymMonCat\rightarrow \Ab$: the Picard group
	\end{itemize}
\item Objects. 
	\begin{itemize}[-]
	\item $\tw{G} \coloneqq \prod_{H \leq G} \Z$
	\item $\twinv{G} \coloneqq \prod_{H \leq G} \left( \sfrac{\Z}{[G:H]} \right)^\times / \{ \pm 1 \}$
	\item $C(G)$: the additive group of functions from the set of conjugacy classes of subgroups of $G$ to $\Z$
	\end{itemize}
	
\end{itemize}

\subsection{Acknowledgments} 

We would like to thank Mike Hill for many insightful conversations concerning the various Picard groups arising throughout this article and for sharing many ideas which ultimately led to the decision to pursue the questions presented here. We would also like to thank Bert Guillou, David Mehrle, {and an anonymous referee} for many helpful conversations and editorial comments.

\section{Picard Groups} 

The Picard group is a fundamental invariant associated to symmetric monoidal categories, consisting of the isomorphism classes of invertible objects. We begin by recalling the notion of Picard groups. See \cite[Chapter XI]{MacLane} for an account of symmetric monoidal categories and strong symmetric monoidal functors. 

\begin{definition}\label{SymMonCat}
Let $(\C, \otimes, \mathbbm{1})$ be a symmetric monoidal category. The \textit{Picard group} of $\C$, denoted $\Pic(\C)$, is the abelian group consisting of isomorphism classes of invertible objects with respect to $\otimes$. An object $X \in \C$ is \textit{invertible, or $\otimes$-invertible,} if there exists an object $Y \in \C$ with $X \otimes Y \cong \1$. The group operation in $\Pic(\C)$ is induced by $\otimes$. 
\end{definition}

The full signature of a symmetric monoidal category includes more than we require in \cref{SymMonCat}, namely the structure of associators, unitors, etc. We suppress this data in the signatures of symmetric monoidal categories throughout and refer to this structure only if explicitly necessary.  

\begin{remark}\label{InvToAut}
An object $X \in (\C, \otimes, \mathbbm{1})$ is $\otimes$-invertible if and only if the endofunctor $(-) \otimes X$ is an automorphism of $\C$. 
\end{remark} 

As the Picard group is defined in terms of the monoidal structure, one would expect that $\Pic$ is functorial for strong symmetric monoidal functors as these functors are precisely those which preserve the structure of symmetric monoidal categories.  

\begin{proposition}\label{PicFunctorial}
Given a strong symmetric monoidal functor $F\colon \C \rightarrow \D$ between symmetric monoidal categories, there is an induced map of abelian groups $F\colon \Pic(\C) \rightarrow \Pic(D)$. 
\end{proposition}

\begin{proof}
Strong symmetric monoidality of a functor implies the preservation of invertibility for objects.
\end{proof}

We will see an application of this fact in the context of equivariant algebra, where several change of groups functors are strong symmetric monoidal functors. 

A class of examples immediately arise from commutative monoids in symmetric monoidal categories. We will denote the category of commutative monoids in a symmetric monoidal category, $\C$, by $\CMon(\C)$. If $\C$ is cocomplete and the functor $R \otimes (-)$ preserves coequalizers for $R \in \CMon(\C)$, then $\Mod_R$ inherits a symmetric monoidal structure via the usual construction using coequalizers. We will denote this symmetric monoidal product by $(-) \otimes_R (-)$. For details, see \cite[Lemma 2.2.2]{HSS}. 

\begin{definition}
Let $(\C, \otimes, \1)$ be a cocomplete, symmetric monoidal category and $R$ be a commutative monoid in $\C$ such that $R \otimes (-)$ preserves coequalizers. The \textit{Picard group} of $R$, denoted $\Pic(R)$, is the Picard group of the symmetric monoidal category $\Mod_R$ of modules over $R$ equipped with $\otimes_R$. 
\end{definition}

For example, one could consider $R \in \CRing \simeq \CMon(\Ab)$ or $\underline{R} \in \Green_G \simeq \CMon(\Mack_G)$, where $\Green_G$ is the symmetric monoidal category of $G$-Green functors. In the case that $R \cong \1$, we have $\Pic(R) \coloneqq \Pic(\Mod_R) \cong \Pic(\C)$, since the unit in any symmetric monoidal category is uniquely a commutative monoid and every object in $\C$ is canonically a module over the unit. 

\begin{remark}\cite[III,4.5]{Hartshorne}
The Picard groups arising from commutative rings in this way can be thought of algebro-geometrically as a higher analog of the group of units. Let $X = \Spec(R)$ with its usual structure sheaf $\scrO_X$. Then,
\begin{eqnarray*}
R^\times &\cong& H^0(X; \scrO^*_X) \\
\Pic(R) &\cong& H^1(X; \scrO^*_X).
\end{eqnarray*} 
\end{remark} 

\begin{example}
The following symmetric monoidal categories have trivial Picard groups: $(\Set, \times, \{*\})$, $(\Mod_\Z,\otimes, \Z)$, $(\Vect_k, \otimes_k, k)$. 
\end{example}

Often, of course, Picard groups are non-trivial. We find many examples in the world of stable homotopy theory. The stable homotopy category $\Ho(\Sp)$ of spectra equipped with its usual smash product and unit given by the sphere spectrum $\Sigma^\infty S^0$ is a symmetric monoidal category. The spectrum $\Sigma^\infty S^1$ is a non-trivial invertible object in $\Ho(\Sp)$ by \cref{InvToAut}, since suspension is an automorphism of the stable homotopy category.  Moreover, every invertible spectrum is weakly equivalent to some (de)suspension of the sphere.  

\begin{proposition}
The Picard group of the stable homotopy category is isomorphic to $\Z$, with generator given by $\Sigma^\infty S^1$, via a generalized degree map. 
\end{proposition}  

\begin{proof}
See the discussion following \cite[Definition 1.1]{HMS}.
\end{proof}

In chromatic homotopy theory, much work has gone into understanding the Picard group of the $K(n)$-local stable homotopy category, where $K(n)$ is Morava K-theory at height $n$, and an approximation of these groups. An introduction to this problem is given in \cite{HMS}, and the aforementioned approximation of these Picard groups is computed in \cite[Theorem A, Theorem B]{BSSW}.

\section{Equivariant Algebra} 
\label{Equivariant Algebra}

Here, we recall the basic notions concerning Mackey functors, change of groups, and modules over the Burnside ring arising from Mackey functors. For more detailed references concerning the notions throughout this section, see \cite{ThevenazWebb}. 

\subsection{Mackey functors}

{Throughout, we denote the conjugate of a subgroup $H \leq G$ by an element $g \in G$ by $H^g \coloneqq gHg^{-1}$. }

\begin{definition}
A \textit{$G$-Mackey functor} is an additive functor $\underline{M} \in [\B^{op}_G, \Ab]$, where $\B_G$ is the \textit{Burnside category} associated to $G$. The Burnside category can be obtained by taking the category of finite $G$-sets with morphisms given by {equivalences classes of} spans of finite $G$-sets and group completing the $\Hom$-monoids. See \cite[Section 2]{ThevenazWebb} for an account of the Burnside category. 
\end{definition}

Equivalently, a Mackey functor $\underline{M} \in \Mack_G$ is the data of 
\begin{itemize}
\item an abelian group $\underline{M}(\sfrac{G}{H})$ for each $H \leq G$
\item maps $\res_J^H\colon \underline{M}(\sfrac{G}{H}) \rightarrow \underline{M}(\sfrac{G}{J})$ for $J \leq H \leq G$, which we call \textit{restriction} maps
\item maps $\tr_J^H\colon \underline{M}(\sfrac{G}{J}) \rightarrow \underline{M}(\sfrac{G}{H})$ for $J \leq H \leq G$, which we call \textit{transfer} maps
\item maps $c_g\colon \underline{M}(\sfrac{G}{H}) \rightarrow \underline{M}(\sfrac{G}{H^g})$ for any $g \in G$, which we call \textit{conjugation} maps
\end{itemize} 
such that the following conditions hold for all $L,J \leq K \leq H \leq G$: 
\begin{enumerate}[(i)]
\item $\res_H^H = \id_{\underline{M}(\sfrac{G}{H})}$
\item $\res_J^K \circ \res_K^H = \res_J^H$
\item $\tr_H^H = \id_{\underline{M}(\sfrac{G}{H})}$
\item $\tr_K^H \circ \tr_J^K = \tr_J^H$
{\item $c_h: \underline{M}(\sfrac{G}{H}) \rightarrow \underline{M}(\sfrac{G}{H^h})$ is the identity map for $h \in H$}
\item $c_g \circ c_k = c_{gk}$ for all $g, k \in G$
\item conjugation maps commute appropriately with restrictions and transfers
\item the double-coset formula is satisfied: 
\[ \res_L^H \tr_J^H = \sum\limits_{JhL \in J \backslash H / L} \tr_{J^{h^{-1}} \cap L}^L \circ c_{h^{-1}} \circ \res_{J \cap L^h}^J. \]
\end{enumerate} 
The conjugation maps yield an action of the Weyl group of $H$ in $G$, denoted $W_G(H)$, on $\underline{M}(\sfrac{G}{H})$. As we mostly restrict our attention to finite abelian groups, it is worth noting that $W_G(H)$ is isomorphic to $\sfrac{G}{H}$ for any subgroup $H \leq G$ in this case. 

This definition allows the use of so-called ``Lewis-diagrams", a diagrammatic expression of the data above which we use in the following examples and often throughout this article. One natural example of a $G$-Mackey functor stems from the Burnside rings, $A(G)$, a class of rings already equipped with the data constituting a Mackey functor. 

\begin{definition}
Let $G$ be a finite group. The \textit{Burnside ring} of $G$, denoted $A(G)$, is the Grothendieck completion of the semiring of isomorphism classes of finite $G$-sets equipped with disjoint union and cartesian product as its operations. A cartesian product of sets equipped with a $G$-action is equipped with the diagonal action by $G$. 
\end{definition}

Recall that every finite $G$-set decomposes as a disjoint of finite, transitive $G$-sets. This gives an isomorphism of abelian groups 
\[ A(G) \cong \bigoplus\limits_{(H) \in \sfrac{\Sub(G)}{\text{conj}}} \Z\{[\sfrac{G}{H}]\}, \]
where the sum is indexed over conjugacy classes of subgroups of $G$. We can give a few examples  for which $A(G)$ has a pleasant ring structure. 

\begin{example}
Let $G = C_p$, the cyclic group of order $p$ for $p$ a prime. Then, 
\[ A(C_p) \cong \sfrac{\Z[x]}{(x^2-px)}, \]
where $\sfrac{C_p}{e}$ corresponds to $x$. 
\end{example}

As mentioned, the Burnside rings $\{A(G)\}_{H \leq G}$ come equipped with natural maps called restrictions and transfers. 

\begin{definition}
Let $J \leq H$ be an inclusion of finite groups. The map 
\[ \res_J^H\colon A(H) \rightarrow A(J) \]
given by restricting $H$-sets to $J$-sets is a ring map. We will refer to this map as the \textit{restriction from $H$ to $J$}. Similarly, we have a map 
\[ \tr_J^H\colon A(J) \rightarrow A(H) \]
given by $X \mapsto H \times_J X$. We will refer to this map as the \textit{transfer from $J$ to $H$}. The transfer is \textit{not} a ring map, but it is an $A(H)$-module map when viewing $A(J)$ as an $A(H)$-module via the restriction map. 

Lastly, let $H, H^g \leq G$ be conjugate subgroups for $g \in G$. There is an isomorphism
\[ c_g\colon A(H) \rightarrow A(H^g) \]
which takes an $H$-set to the $H^g$-set given by pulling back along the isomorphism $H^g \rightarrow H$ given by conjugation by $g^{-1}$.  
\end{definition}

\begin{definition}
The \textit{Burnside Mackey functor} for $G$, denoted $\underline{A}$ is the Mackey functor given by 
\begin{itemize}
\item $\underline{A}(\sfrac{G}{H}) \coloneqq A(H)$ 
\item $\res_J^H\colon A(H) \rightarrow A(J)$ given by restriction of $H$-sets
\item $\tr_J^H\colon A(J) \rightarrow A(H)$ given by the transfer of $J$-sets
\item $c_g\colon A(H) \rightarrow A(H^g)$ is given by conjugation of $H$-sets. 
\end{itemize}
\end{definition}

\begin{example} 
Let $G=C_p$. The Burnside Mackey functor for $C_p$ has the following Lewis diagram: 
\[ \SigmaTwoMack{$A(C_p) \cong \Z\{\sfrac{C_p}{e}, \sfrac{C_p}{C_p}\}$}{$\begin{mylittlematrix}{p & 1}\end{mylittlematrix}$}{$A(e) \cong \Z.$}{$\begin{mylittlematrix}{1 \\ 0}\end{mylittlematrix}$} \]
Here, we use matrices to denote maps between free abelian groups with ordered basis as depicted. 
\end{example}

\begin{example}
Let $G = C_2$. Given a $\Z$-module $M$, we can define the \textit{constant $C_2$-Mackey functor at $M$}, denoted $\underline{M}$, via the Lewis-diagram
\[ \SigmaTwoMack{$M$}{$1$}{$M.$}{$2$} \]
The restriction and transfer denote multiplication by 1 (the identity) and multiplication by 2, respectively, while the conjugation maps are all trivial. This definition can be extended to finite groups by defining $\underline{M}(\sfrac{G}{H}) \coloneqq M$, all restrictions to be the identity, and $\tr_J^H$ to be multiplication by $[H:J]$. The conjugation maps are all defined to be the identity.  

The \textit{dual constant $C_2$-Mackey functor at $M$}, denoted $\underline{M}^*$ is given by 
\[ \SigmaTwoMack{$M$}{$2$}{$M,$}{$1$} \]
where again the conjugation maps are trivial. Similarly, dual constant Mackey functors can be defined for general finite groups. 
\end{example}

While we have not discussed morphisms of Mackey functors, they can be easily described via Lewis diagrams. A morphism $\varphi\colon \underline{M} \rightarrow \underline{N}$ between Mackey functors is a collection of $W_G(H)$-equivariant maps $\{\varphi_H\colon \underline{M}(\sfrac{G}{H}) \rightarrow \underline{N}(\sfrac{G}{H})\}_{H \leq G}$ that commute with restrictions and transfers independently. The resulting category, $\Mack_G$, of $G$-Mackey functors can be given the structure of a symmetric monoidal product via a Day convolution construction.

\begin{definition}
Let $\underline{M}$ and $\underline{N}$ be $G$-Mackey functors. The \textit{box product} of $\underline{M}$ and $\underline{N}$, denoted $\underline{M} \boxtimes \underline{N}$, is the left Kan extension of the pointwise-tensor of $\underline{M}$ and $\underline{N}$ along the cartesian product of $G$-sets: 
\begin{center}
\begin{tikzpicture}
\node (00) at (0,0) {$\B_G^{op} \times \B_G^{op}$};
\node (0-2) at (0,-2) {$\B_G^{op}$}; 
\node (30) at (3,0) {$\Ab$};

\draw[->] (00) to node[above] {$\underline{M} \otimes \underline{N}$} (30); 
\draw[->] (00) to node[left] {$\times$} (0-2);
\draw[->, dashed] (0-2) to node[below right] {$\underline{M} \boxtimes \underline{N}$} (30); 
\end{tikzpicture}
\end{center}
\end{definition}

{One can verify that the category of $G$-Mackey functors equipped with the box product is a symmetric monoidal category with unit given by the Burnside Mackey functor. See, for example, \cite[Lemma 2.23]{Shulman}. 
}

Unpacking the definition of the box product allows for a description of maps $\underline{M} \boxtimes \underline{N} \rightarrow \underline{P}$ for arbitrary Mackey functors $\underline{M}, \underline{N}, \underline{P} \in \Mack_G$ \cite[Lemma 2.17]{Shulman}. We will use a slightly improved version of this universal property and borrow notation from \cite{Shulman}.

\begin{lemma}
\label{MapsOutOfBoxProduct}
Let $G$ be a finite group and $\underline{M}$, $\underline{N}$, and $\underline{P}$ be $G$-Mackey functors. The data of a map $\theta\colon \underline{M} \boxtimes \underline{N} \rightarrow \underline{P}$ is equivalent to a collection of maps $\{\theta_H\colon \underline{M}(\sfrac{G}{H}) \otimes \underline{N}(\sfrac{G}{H}) \rightarrow \underline{P}(\sfrac{G}{H})\}_{H \leq G}$ such that the following diagrams commute for all $J \leq H \leq G$ and $g \in W_G(H)$: 
\[
\begin{tikzpicture}
\node (00) at (0,0) {$\underline{M}(\sfrac{G}{J}) \otimes \underline{N}(\sfrac{G}{J})$};
\node (03) at (0,3) {$\underline{M}(\sfrac{G}{H}) \otimes \underline{N}(\sfrac{G}{H})$};
\node (40) at (4,0) {$\underline{P}(\sfrac{G}{J})$};
\node (43) at (4,3) {$\underline{P}(\sfrac{G}{H})$};

\draw[->] (03) to node[left] {$\res_J^H \otimes \res_J^H$} (00);
\draw[->] (00) to node[below] {$\theta_J$} (40);

\draw[->] (03) to node[above] {$\theta_H$} (43);
\draw[->] (43) to node[right] {$\res_J^H$} (40);
\end{tikzpicture}
\qquad 
\begin{tikzpicture}
\node (00) at (0,0) {$\underline{M}(\sfrac{G}{H}) \otimes \underline{N}(\sfrac{G}{H})$};
\node (03) at (0,3) {$\underline{M}(\sfrac{G}{J}) \otimes \underline{N}(\sfrac{G}{J})$};
\node (40) at (4,0) {$\underline{P}(\sfrac{G}{H})$};
\node (43) at (4,3) {$\underline{P}(\sfrac{G}{J})$};
\node (-11) at (-1.5,1.5) {$\underline{M}(\sfrac{G}{J}) \otimes \underline{N}(\sfrac{G}{H})$};

\draw[->] (00) to node[below] {$\theta_H$} (40);
\draw[->] (03) to node[above] {$\theta_J$} (43);
\draw[->] (43) to node[right] {$\tr_J^H$} (40);

\draw[->] (-11) to node[left=1ex] {$\id \otimes \res_J^H$} (03);
\draw[->] (-11) to node[left=1ex] {$\tr_J^H \otimes \id$} (00);
\end{tikzpicture}
\]
\[
\begin{tikzpicture}
\node (00) at (0,0) {$\underline{M}(\sfrac{G}{H}) \otimes \underline{N}(\sfrac{G}{H})$};
\node (03) at (0,3) {$\underline{M}(\sfrac{G}{H}) \otimes \underline{N}(\sfrac{G}{H})$};
\node (40) at (4,0) {$\underline{P}(\sfrac{G}{H})$};
\node (43) at (4,3) {$\underline{P}(\sfrac{G}{H})$};

\draw[->] (03) to node[left] {$c_g \otimes c_g$} (00);
\draw[->] (00) to node[below] {$\theta_H$} (40);

\draw[->] (03) to node[above] {$\theta_H$} (43);
\draw[->] (43) to node[right] {$c_g$} (40);
\end{tikzpicture}
\qquad
\begin{tikzpicture}
\node (00) at (0,0) {$\underline{M}(\sfrac{G}{H}) \otimes \underline{N}(\sfrac{G}{H})$};
\node (03) at (0,3) {$\underline{M}(\sfrac{G}{J}) \otimes \underline{N}(\sfrac{G}{J})$};
\node (40) at (4,0) {$\underline{P}(\sfrac{G}{H})$.};
\node (43) at (4,3) {$\underline{P}(\sfrac{G}{J})$};
\node (-11) at (-1.5,1.5) {$\underline{M}(\sfrac{G}{H}) \otimes \underline{N}(\sfrac{G}{J})$};

\draw[->] (00) to node[below] {$\theta_H$} (40);
\draw[->] (03) to node[above] {$\theta_J$} (43);
\draw[->] (43) to node[right] {$\tr_J^H$} (40);

\draw[->] (-11) to node[left=1ex] {$\res_J^H \otimes \id$} (03);
\draw[->] (-11) to node[left=1ex] {$\id \otimes \tr_J^H$} (00);
\end{tikzpicture}
\]
\end{lemma}

\begin{proof}
An analogous result is given by \cite[Lemma 2.17]{Shulman}, where the orbits and maps appearing in the statement are replaced by arbitrary finite \(G\)-sets \(\mathbf{b}\) and \(\mathbf{c}\) and equivariant maps \(f\colon \mathbf{b} \to \mathbf{c}\). The desired result follows from leveraging the proof of \cite[Lemma 2.17]{Shulman} with the additivity of Mackey functors.
\end{proof}

\begin{definition}
Let $\underline{M}, \underline{N}$, and $\underline{P}$ be $G$-Mackey functors. Given a map $\theta\colon \underline{M} \boxtimes \underline{N} \rightarrow \underline{P}$, the \textit{Dress pairing} associated to $\theta$ is the collection of maps $\{\theta_H \colon \underline{M}(\sfrac{G}{H}) \otimes \underline{N}(\sfrac{G}{H}) \rightarrow \underline{P}(\sfrac{G}{H})\}_{H \leq G}$ described in \cref{MapsOutOfBoxProduct}. We will denote by $\Dress{\underline{M}}{\underline{N}}{\underline{P}}$ the set of Dress pairings of maps $\underline{M} \boxtimes \underline{N} \rightarrow \underline{P}$.  
\end{definition}

\subsection{Change of Groups}
Throughout this section, we let $G$ denote a finite group. Further, let $H \leq G$ and $N \trianglelefteq G$ with $Q = \sfrac{G}{N}$. There are several adjunctions relating the category of $G$-Mackey functors to the categories of $Q$-Mackey functors and $H$-Mackey functors which allow us to inductively approach a computation of $\Pic(\Mack_G)$.  We begin by recalling the usual induction-restriction adjunction. 

\begin{prop} 
There is an adjunction 
\[
\adjunction{$\Mack_G$}{$\Mack_H,$}{$\uparrow_H^G$}{$\downarrow_H^G$}
\]
where $\uparrow_H^G(\underline{M})(\sfrac{G}{J}) \coloneqq \underline{M}(\res_H^G(\sfrac{G}{J}))$ and $\downarrow_H^G(\underline{N})(\sfrac{H}{K}) \coloneqq \underline{N}(\sfrac{G}{K})$ for $\underline{M} \in \Mack_H$ and $\underline{N} \in \Mack_G$. The restrictions, transfers, and conjugations in both $\uparrow_H^G \underline{M}$ and $\downarrow_H^G \underline{N}$ are induced by those in $\underline{M}$ and $\underline{N}$, respectively. We will refer to $\uparrow_H^G$ and $\downarrow_H^G$ as the \textit{induction} and \textit{restriction} of Mackey functors, respectively. 
\end{prop}

\begin{proof}
The induction and restriction of Mackey functors are obtained by pre-composing with the restriction and induction functors on Burnside categories, respectively, and these functors on Burnside categories are appropriately adjoint so as to induce the desired adjunction. 
\end{proof}

\begin{example}
The induction of $\Z \in \Ab \simeq \Mack_e$ to the cyclic group of order two, $C_2$, has the following Lewis diagram: 
\[ \SigmaTwoMack{$\Z$}{$\Delta$}{$\Z[C_2].$}{$\nabla$} \]
Here, $\Z[C_2]$ is the free abelian group with generators given by the elements of $C_2$ and $C_2$-action given by multiplication in $C_2$. Then, $\Delta$ is the diagonal map which picks out the sum of the basis elements and $\nabla$ is the ``fold" map which sends every basis element to $1$. 
\end{example}

This examples demonstrates the more general failure of induction of Mackey functors to be strong symmetric monoidal for $H \lneq G$ as it does not preserve the unit up to isomorphism. On the other hand, restriction of Mackey functors is always a strong symmetric monoidal functor. In particular, this implies that the restriction of Mackey functors induces a map 
\[ \downarrow_H^G\colon \Pic(\Mack_G) \rightarrow \Pic(\Mack_H). \]
as in \cref{PicFunctorial}. 

The restriction of Mackey functors also has a right adjoint, \textit{coinduction}, which agrees with induction. This generalizes the same phenomenon seen in representations of finite groups. We now turn our view towards a few adjunctions which witness the close relationship between $\Mack_G$ and $\Mack_{\sfrac{G}{N}}$, where $N$ is a normal subgroup of $G$. The change of groups between $G$ and its various quotients behaves much more nicely concerning invertibility.

\begin{definition}
Let $q^*\colon \B_Q^{op} \rightarrow \B_G^{op}$ denote the functor between Burnside categories given by viewing $Q$-sets and spans of $Q$-sets as $G$-sets and spans of $G$-sets, respectively, via pulling back along the quotient map $q\colon G \twoheadrightarrow Q {= \sfrac{G}{N}}$. We get a functor $\qRes_Q^G\colon \Mack_G \rightarrow \Mack_Q$ given by precomposition with $q^*$ which we will refer to as \textit{restriction along the quotient from $G$ to $Q$}. This functor is denoted by $\beta^!$ in \cite[Section 5]{ThevenazWebb} and goes by the name ``deflation" in various places throughout the literature. 
\end{definition} 

Conceptually, the restriction along a quotient $G \twoheadrightarrow \sfrac{G}{N}$ identifies the subgroups of $\sfrac{G}{N}$ with those in $G$ containing $N$ and simply forgets the data of $\underline{M}$ for subgroups not containing {$N$}. We find the following visualization of the restriction along a quotient useful: 

\begin{heuristic} 
The restriction along $q: G \twoheadrightarrow \sfrac{G}{N}$ of a Mackey functor $\underline{M} \in \Mack_G$ to $\qRes_{\sfrac{G}{N}}^G \underline{M} \in \Mack_{\sfrac{G}{N}}$ can be viewed as follows: 
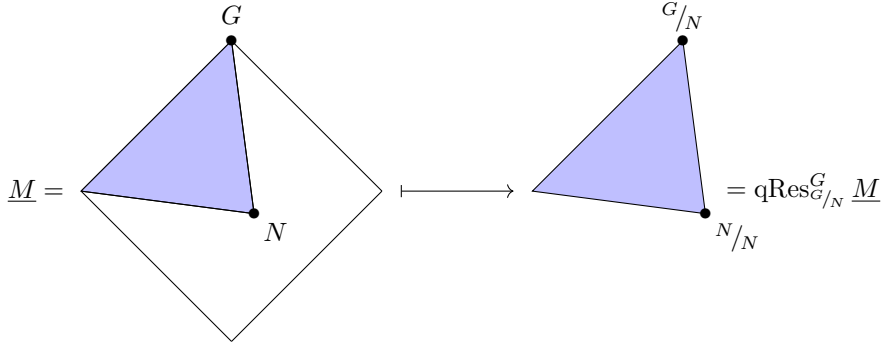
\[
\begin{tikzpicture}
\node at (-2.6,0) {$\underline{M} = $};

\draw[-] (0,2) to (-2,0);
\draw[-] (-2,0) to (0,-2);
\draw[-] (0,-2) to (2,0);
\draw[-] (2,0) to (0,2);

\draw[fill=blue!25] (-2,0) -- (0.3,-0.3) -- (0,2) -- cycle;

\node at (0.3,-0.3) {$\bullet$}; 
\node[below right] at (0.3,-0.3) {$N$};

\node at (0,2) {$\bullet$}; 
\node[above] at (0,2.1) {$G$};

\draw[-] (-2,0) to (0.3,-0.3);
\draw[-] (0,2) to (0.3,-0.3);

\draw[|->] (2.25,0) to (3.75,0);

\draw[fill=blue!25] (4,0) -- (6.3,-0.3) -- (6,2) -- cycle;

\node at (6.3,-0.3) {$\bullet$}; 
\node[below right] at (6.3,-0.3) {$\sfrac{N}{N}$};

\node at (6,2) {$\bullet$}; 
\node[above] at (6,2) {$\sfrac{G}{N}$};

\node at (7.6,0) {$= \qRes_{\sfrac{G}{N}}^G \underline{M}$};
\end{tikzpicture} 
\]
We suppress the Weyl actions in the above pictures for aesthetic reasons, but note that the actions of various Weyl groups in $\qRes_{\sfrac{G}{N}}^G \underline{M}$ follow from the isomorphisms $\sfrac{H}{J} \cong \sfrac{\sfrac{H}{N}}{\sfrac{J}{N}}$ for $N \leq J \leq H \leq G$. 
\end{heuristic} 

The restriction of Mackey functors along a quotient map has a strong symmetric monoidal left adjoint which we will denote by $\qInd_Q^G$. 

\begin{definition}
The \textit{induction of Mackey functors along a quotient map} $q\colon G \twoheadrightarrow Q$ is the functor  
\[ \qInd_Q^G (-) \coloneqq \Lan_{q^*} (-)\colon \Mack_Q \rightarrow \Mack_G. \]
This functor goes by the name of $\beta_!$ in \cite[Section 5]{ThevenazWebb}. We use the term ``induction along a quotient" in light of this functor being left adjoint to restriction along a quotient as in \cref{qIndResAdj}.  
\end{definition}

{Since the induction of $Q$-Mackey functors along a quotient $q: G \twoheadrightarrow Q = \sfrac{G}{N}$ is defined by a left Kan extension, the usual coend formula can be unpacked to give a concrete description of this construction. A fully general formula is not necessary in what follows, but we give a rough description of inducing along a quotient following \cref{def:inflation}. Before this, we prove a formal property of $\qInd_Q^G$ which gives rise to a family of examples of $G$-Mackey functors which are induced from $Q$-Mackey functors.
}

\begin{proposition}
\label{prop:qIndSymMon}
The induction of Mackey functors along a quotient map is a strong symmetric monoidal functor $\Mack_Q \rightarrow \Mack_G$. 
\end{proposition}

\begin{proof}
Since $q^*$ is a product preserving functor, the left Kan extension of a Mackey functor along $q^*$ is an product preserving functor $\B_G^{op} \rightarrow \Ab$, hence a Mackey functor. {See \cite[Theorem 1.5(iii)]{BD}}. Left Kan extending along $q^*$ is a {strong} symmetric monoidal functor for the box product as the box product is the Day convolution product of the Cartesian product with the tensor product of abelian groups and $q^*$ is product preserving {\cite[Proposition 1]{DS}}. That is, $\qInd_Q^G$ is strong symmetric monoidal via commuting left Kan extensions.
\end{proof}

\begin{prop}\label{qIndResAdj}
There is an adjunction 
\[
\adjunction{$\Mack_Q$}{$\Mack_G$}{$\qInd_Q^G$}{$\qRes_Q^G$}
\]
\end{prop}

\begin{proof}
This follows from defining $\qInd_Q^G$ as the left Kan extension along $q^*$, while $\qRes_Q^G$ is defined as precomposition with $q^*$. 
\end{proof}

{
\begin{example}
Given a quotient map $q: G \twoheadrightarrow Q = \sfrac{G}{N}$ of groups, the induction of $Q$-Mackey functors to $G$-Mackey functors preserves invertible objects by \cref{prop:qIndSymMon}. In particular, the induction of the Burnside Mackey functor for $Q$ is isomorphic to the Burnside Mackey functor for $G$. More generally, the induction of any invertible $Q$-Mackey functor is explicitly described in \cref{qIndOfTwisted}.

In the case of the canonical quotient $C_2 \twoheadrightarrow e$, the Burnside Mackey functor for the trivial group is given by $A(e) \cong \Z$ and $\qInd_e^{C_2} \Z$ is isomorphic to the Burnside Mackey functor for $C_2$. This example demonstrates a more general phenomenon that, roughly speaking, the difference between inducing along a quotient and inflation as in \cref{def:inflation} is the formal addition of transfers and restrictions. 

In the aforementioned case, one can calculate first the inflation of $\Z \in \Ab$ to the $C_2$-Mackey functor $\Inf_e^{C_2} \Z$ which vanishes at $\sfrac{C_2}{e}$ and evaluates to $\Z$ at $\sfrac{C_2}{C_2}$. Then, formally adding the restriction of $1 \in \Z = \Inf_e^{C_2} \Z(\sfrac{C_2}{C_2})$ and the transfer of this restriction yields the Burnside Mackey functor for $C_2$. The remaining data is forced by the axioms of a Mackey functor.
\end{example}

\begin{example}
Consider the quotient map $q: C_4 \twoheadrightarrow C_2$. The induction of the constant $C_2$-Mackey functor at $\Z$ along $q$ is the $C_4$-Mackey functor with the following Lewis diagram: 
\[ 
\begin{tikzpicture}
\node at (-3,2) {$\qInd_{C_2}^{C_4} \underline{\Z} =$};

\node (C4) at (0,4) {$\Z \oplus \Z$};
\node (C2) at (0,2) {$\Z \oplus \Z$};
\node (e) at (0,0) {$\Z$};

\draw[->, bend right] (C4) to node[left] {$\mylittlematrix{1 & 0 \\ 0 & 2}$} (C2);
\draw[->, bend right] (C2) to node[right] {$\mylittlematrix{2 & 0 \\ 0 & 1}$} (C4);

\draw[->, bend right] (C2) to node[left] {$\mylittlematrix{1 & 2}$} (e);
\draw[->, bend right] (e) to node[right] {$\mylittlematrix{0 \\ 1}$} (C2);
\end{tikzpicture}
\]
This computation can be verified using \cref{qIndResAdj}: the Lewis diagram above depicts a Mackey functor $\underline{M}$ which is generated at the top level. Then, we have a natural isomorphism
\[ \Hom(\underline{M}, \underline{N}) \cong \{ n \in \underline{N}(\sfrac{C_4}{C_2}) \mid \tr_{C_2}^{C_4} \res_{C_2}^{C_4} n = 2n\} \cong \Hom(\underline{\Z}, \qRes_{C_2}^{C_4} \underline{N}) \cong \Hom(\qInd_{C_2}^{C_4} \underline{\Z}, \underline{N}) \]
for any $\underline{N} \in \Mack_{C_4}$ and the Yoneda Lemma identifies $\qInd_{C_2}^{C_4} \underline{\Z}$ with $\underline{M}$. 
\end{example}
%
%
%
}

We also consider the \textit{inflation of $Q$-Mackey functors along the quotient map $q$}. This is a construction that produces $G$-Mackey functors and is closely related to the induction along a quotient map via a left inverse of the functor $q^*$. 

\begin{definition}
Let $\Fix_N\colon \B_G^{op} \rightarrow \B_Q^{op}$ denote the functor given on $G$-sets and spans of $G$-sets by taking $N$-fixed points. 
\end{definition} 

\begin{proposition}
The functor $\Fix_N$ gives a left inverse of $q^*$. 
\end{proposition} 

\begin{proof}
The $N$-fixed points of a $Q$-set $T$ viewed as a $G$-set by pulling back along the quotient map are given by $T$. 
\end{proof}

\begin{definition}
\label{def:inflation}
The \textit{inflation of Mackey functors} from $Q$ to $G$ is the functor $\Inf_Q^G\colon \Mack_Q \rightarrow \Mack_G$ given by precomposition with $\Fix_N$.  
\end{definition}

Again, we describe the inflation functor as identifying the subgroups of $\sfrac{G}{N}$ with those of $G$ containing $N$ and extending $\underline{M}$ to a $G$-Mackey functor by defining $\Inf_Q^G \underline{M}(\sfrac{G}{H}) = 0$ for all subgroups $H \ngeq N$. We find a similar visualization of this functor useful: 

\begin{heuristic}
The inflation of a Mackey functor $\underline{M} \in \Mack_Q$ along the quotient $q$ can be viewed as follows:  
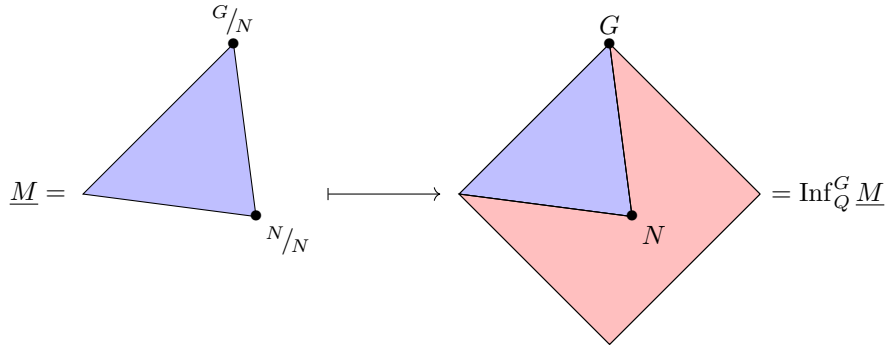
\[
\begin{tikzpicture}
\node at (-1.6,0) {$\underline{M} =$};

\draw[-] (6,2) to (4,0);
\draw[-] (4,0) to (6,-2);
\draw[-] (6,-2) to (8,0);
\draw[-] (8,0) to (6,2);

\draw[fill=blue!25] (4,0) -- (6.3,-0.3) -- (6,2) -- cycle;
\draw[fill=red!25] (6,2) -- (6.3,-0.3) -- (4,0) -- (6,-2) -- (8,0) -- cycle;

\node at (6.3,-0.3) {$\bullet$}; 
\node[below right] at (6.3,-0.3) {$N$};

\node at (6,2) {$\bullet$}; 
\node[above] at (6,2) {$G$};

\draw[-] (4,0) to (6.3,-0.3);
\draw[-] (6,2) to (6.3,-0.3);

\draw[|->] (2.25,0) to (3.75,0);

\draw[fill=blue!25] (-1,0) -- (1.3,-0.3) -- (1,2) -- cycle;

\node at (1.3,-0.3) {$\bullet$}; 
\node[below right] at (1.3,-0.3) {$\sfrac{N}{N}$};

\node at (1,2) {$\bullet$}; 
\node[above] at (1,2) {$\sfrac{G}{N}$};

\node at (8.9,0) {$= \Inf_Q^G \underline{M}$};
\end{tikzpicture} 
\]
where the red (lower right) region of $\Inf_Q^G \underline{M}$ is interpreted as being zero. Again, we suppress the Weyl actions in this picture.
\end{heuristic}

\begin{definition}
The \textit{$N$-geometric fixed points} functor, denoted $\Phi^N$, is given by 
\[ \Phi^N (-) \coloneqq \Lan_{\Fix_N} (-)\colon \Mack_G \rightarrow \Mack_Q. \]
This functor is denoted by $(-)^+$ in \cite[Section 2]{ThevenazWebb}. 
\end{definition}

{Because the $N$-geometric fixed points functor is defined as a left Kan extension, the usual coend formula can be unpackaged to give a levelwise description of $\Phi^N \underline{M}$ for any $G$-Mackey functor $\underline{M}$.}

\begin{remark}\label{GeomFixPtsFormula}
Let $\underline{M} \in \Mack_G$. The $N$-geometric fixed points Mackey functor $\Phi^N \underline{M}$ is a $\sfrac{G}{N}$-Mackey functor with
\[ (\Phi^N \underline{M})(\sfrac{\sfrac{G}{N}}{\sfrac{H}{N}}) \cong \underline{M}(\sfrac{G}{H}) \left/ \left\langle \bigcup\limits_{\substack{J \leq H \\ J \ngeq N}} \im(\tr_J^H) \right\rangle \right. . \]
This can be found in the discussion preceding \cite[Proposition 2.3]{ThevenazWebb}. 

{Alternatively, geometric fixed points is described as $\qRes_Q^G(\widetilde{\underline{E\mathcal{F}}}_N \boxtimes (-))$ in \cite[Definition 5.3.8]{HMQ}, where $\widetilde{\underline{E\mathcal{F}}}_N$ is as in \cite[Definition 5.2.1]{HMQ}.
}
\end{remark}

{The name ``geometric fixed points" is inspired by language used in equivariant stable homotopy theory. There is a symmetric monoidal functor $\Sigma_{G,+}^\infty (-)$ which takes a topological space equipped with a $G$-action and produces a $G$-spectrum. However, there are multiple notions of fixed points in the equivariant stable homotopy category, and not all of these notions are compatible with $\Sigma_{G,+}^\infty$. 

For $N \lhd G$ a normal subgroup with $Q = \sfrac{G}{N}$, the $N$-geometric fixed point spectrum of a $G$-spectrum $E$ is a $Q$-spectrum $\Phi^N E$ which satisfies 
\[ \Phi^N \Sigma_{G,+}^\infty X \simeq \Sigma_{Q,+}^\infty (X^N) \in \Ho(\Sp^{Q}). \] 
Note that the $N$-fixed points of $X$ can be viewed as retaining an action by $Q$. Geometric fixed points can be modeled as taking the smash product with the $G$-spectrum $\Sigma_{G,+}^\infty \widetilde{E\mathcal{F}}_N$ and computing the ``categorical $N$-fixed point spectrum." The $G$-space $\widetilde{E\mathcal{F}}_N$ is the mapping cone of $(E\mathcal{F}_N)_+ \rightarrow S^0$, where $E\mathcal{F}_N$ is determined up to equivariant homotopy equivalence by the following properties: 
\begin{enumerate}[i)]
\item $E\mathcal{F}_N^H \simeq \emptyset$ when $N \leq H$, 
\item $E\mathcal{F}_N^H \simeq *$ otherwise.
\end{enumerate} 

One can verify that $\underline{\pi}_0(\Sigma_{G,+}^\infty \widetilde{E\mathcal{F}}_N) \cong \widetilde{\underline{E\mathcal{F}}}_N$, and the geometric fixed points of a $G$-Mackey functor $\underline{M}$ can be recovered as 
\begin{equation} 
\label{GeomFixedPointsTop} 
\Phi^N(\underline{M}) = \qRes_Q^G (\underline{M} \boxtimes \widetilde{\underline{E\mathcal{F}}}_N) = \qRes_Q^G \underline{\pi}_0(H\underline{M} \wedge \Sigma_{G,+}^\infty \widetilde{E\mathcal{F}}_N) = \qRes_Q^G \left(\underline{M} \boxtimes \underline{\pi}_0(\Sigma_{G,+}^\infty \widetilde{E\mathcal{F}}_N)\right),
\end{equation}
where $H(-): \Mack_G \rightarrow \Ho(\Sp^G)$ is the embedding of $G$-Mackey functors into $G$-spectra which produces an Eilenberg-Mac Lane $G$-spectrum for any Mackey functor. In this way, the geometric fixed points functor in equivariant algebra is an algebraic shadow of the geometric fixed point spectrum functor in equivariant stable homotopy theory.  
}

\begin{prop}
There is an adjunction 
\[
\adjunction{$\Mack_G$}{$\Mack_Q$}{$\Phi^N$}{$\Inf_Q^G$}
\]
\end{prop}

\begin{proof}
Again, $\Phi^N$ is defined by left Kan extending along $\Fix_N$. 
\end{proof}

\begin{prop}
The $N$-geometric fixed points functor is strong symmetric monoidal. 
\end{prop}

\begin{proof}
The left Kan extension along a product and coproduct preserving functor between Burnside categories gives a strong symmetric monoidal functor on Mackey functors. It suffices to check that $\Fix_N$ is such a functor. 
\end{proof}

Again, we find the following visualization useful: 

\begin{heuristic}
The $N$-geometric fixed points of a Mackey functor $\underline{M} \in \Mack_G$ can be visualized as follows:
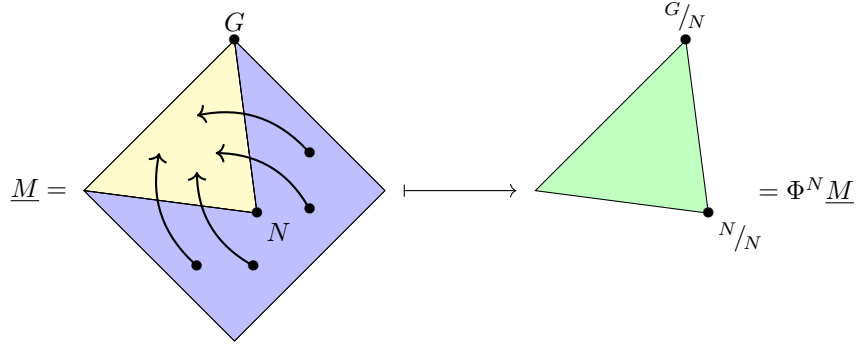
\[
\begin{tikzpicture}
\node at (-2.6,0) {$\underline{M} = $};

\draw[-] (0,2) to (-2,0);
\draw[-] (-2,0) to (0,-2);
\draw[-] (0,-2) to (2,0);
\draw[-] (2,0) to (0,2);

\draw[fill=yellow!25] (-2,0) -- (0.3,-0.3) -- (0,2) -- cycle;
\draw[fill=blue!25] (0,2) -- (0.3,-0.3) -- (-2,0) -- (0,-2) -- (2,0) -- cycle;

\node at (0.3,-0.3) {$\bullet$}; 
\node[below right] at (0.3,-0.3) {$N$};

\node at (0,2) {$\bullet$}; 
\node[above] at (0,2) {$G$};

\draw[-] (-2,0) to (0.3,-0.3);
\draw[-] (0,2) to (0.3,-0.3);

\node at (0.25,-1) {$\bullet$};
\node at (-0.5,-1) {$\bullet$};
\node at (1,-0.25) {$\bullet$};
\node at (1,0.5) {$\bullet$};

\draw[->, bend left, thick] (0.25,-1) to ( -0.5,0.25);
\draw[->, bend left, thick] (-0.5,-1) to ( -1,0.5);

\draw[->, bend right, thick] (1,-0.25) to (-0.25,0.5);
\draw[->, bend right, thick] (1,0.5) to (-0.5,1);

\draw[|->] (2.25,0) to (3.75,0);

\draw[fill=green!25] (4,0) -- (6.3,-0.3) -- (6,2) -- cycle;

\node at (6.3,-0.3) {$\bullet$}; 
\node[below right] at (6.3,-0.3) {$\sfrac{N}{N}$};

\node at (6,2) {$\bullet$}; 
\node[above] at (6,2) {$\sfrac{G}{N}$};

\node at (7.6,0) {$= \Phi^N \underline{M}$};
\end{tikzpicture}, 
\]
where the arrows in the diagram on the left represent transfers from subgroups $J \ngeq N$ to subgroups $H \geq N$. The levelwise values in the right diagram represent corresponding values in the yellow (top left) region of the left diagram modulo the image of depicted transfers. We similarly suppress the Weyl actions in this picture.
\end{heuristic}

With this description of geometric fixed points, we immediately deduce the following: 

\begin{prop}
The functor $\Phi^N$ is left inverse to $\Inf_Q^G$.  
\end{prop}

Formally, we can immediately deduce another right inverse of geometric fixed points: 

\begin{prop}\label{PhiLeftInvToqInd}
The functor $\Phi^N$ is left inverse to $\qInd_Q^G$. Therefore, $\qInd_Q^G$ gives a strong symmetric monoidal embedding of $\Mack_Q$ into $\Mack_G$.
\end{prop}

\begin{proof}
We observe that 
\[ (\Phi^N \circ \qInd_Q^G)(\underline{M}) = \Lan_{\Fix_N}(\Lan_{q^*} \underline{M}) = \Lan_{\Fix_N \circ q^*} \underline{M} = \Lan_{\id} \underline{M} = \underline{M}. \]
\end{proof}

\begin{remark}
To demonstrate the difference between $\qInd_Q^G$ and $\Inf_Q^G$, we note that $\Inf_{\sfrac{G}{G}}^G \Z$ is a Mackey functor which vanishes everywhere except for $\sfrac{G}{G}$ while $\qInd_{\sfrac{G}{G}}^G \Z$ is the Burnside Mackey functor. 
\end{remark}

\cref{PhiLeftInvToqInd} and \cref{PicFunctorial} yield a description of $\Pic(\Mack_G)$ in terms of $\Pic(\Mack_Q)$ for any quotient $Q$ of $G$. 

\begin{corollary}
\label{PicSplitting}
Let $N \leq G$ be a normal subgroup of $G$ and $Q \coloneqq \sfrac{G}{N}$. There is a split exact sequence
\[ 0 \rightarrow \Pic(\Mack_Q) \xrightarrow{\qInd_Q^G} \Pic(\Mack_G) \rightarrow \coker(\qInd_Q^G) \rightarrow 0, \]
where the inclusion $\Pic(\Mack_Q) \xrightarrow{\qInd_Q^G} \Pic(\Mack_G)$ has a retraction given by $\Phi^N$. In particular, we get a splitting 
\[ \Pic(\Mack_G) \cong \Pic(\Mack_Q) \oplus \ker(\Phi^N). \]
\end{corollary}
We will leverage this splitting to prove the main results of \cref{Classification Section}.

\section{The Picard Group of $G$-Spectra}
\subsection{K{\"u}nneth Objects and Dualizable Objects}

Throughout this section, $(\C, \otimes, \mathbbm{1})$ will denote a closed symmetric monoidal category. We begin by recalling some finiteness conditions related to dualizability, then specialize to the (equivariant) algebraic setting of the category of modules over some sufficiently nice commutative monoid object. For more details concerning these notions, see \cite{FLM}. 

\begin{definition}
An object $X \in \C$ is \textit{dualizable} if there exists an object $Y \in \C$ and maps
\begin{enumerate}[(i)]
\item $\eta\colon \mathbbm{1} \rightarrow X \otimes Y$ 
\item $\epsilon\colon Y \otimes X \rightarrow \mathbbm{1}$
\end{enumerate}
called the \textit{unit} and \textit{counit}, respectively, which together satisfy the \textit{triangle identities:}
\begin{enumerate}[(a)]
\item $({\id_X} \otimes \epsilon) \circ (\eta \otimes {\id_X}) = \id_X$
\item $(\epsilon \otimes {\id_Y}) \circ ({\id_Y} \otimes \eta) = \id_Y$.
\end{enumerate}
Such an object $Y$ is called a \textit{dual of $X$}, and it is unique up to a unique isomorphism preserving units and counits.
\end{definition}

This definition is valid more generally in any symmetric monoidal category. However, in \textit{closed} symmetric monoidal categories, it is particularly well-behaved: if \(X\) is dualizable, then it has a canonical choice of dual given by \(DX \coloneqq [X, \mathbbm{1}]\), as well as a canonical counit given by the evaluation map \([X, \mathbbm{1}] \otimes X \to \mathbbm{1}\). In other words, \(X\) is dualizable if and only if there is a unit map \(\eta\colon \mathbbm{1} \to X \otimes DX\) which satisfies the triangle identities with the evaluation map.

\begin{definition}
A dualizable object $X$ in a symmetric monoidal category is called a \textit{K{\"u}nneth object} if it is a retract of a finite copower of the unit object. That is, $X$ is a retract of $\coprod_{i=1}^n \mathbbm{1}$ for some natural number $n$. 
\end{definition} 

In the category of modules over a commutative ring, K\"unneth is equivalent to dualizable, which is equivalent to finitely generated and projective. However, this is not the case in general. For example, the Mackey functors $\uparrow_H^G \underline{A} \in \Mack_G$ for $H \lneq G$ are dualizable (i.e., finitely generated and projective), but they are not K\"unneth objects. 

From here, we specialize to the context of additive symmetric monoidal categories. Moreover, we borrow notation from \cite{FLM}, letting $k\C \subseteq d\C$ denote the full subcategories of K{\"u}nneth objects and dualizable objects in $\C$, respectively. These are additive symmetric monoidal subcategories of $\C$ with strong symmetric monoidal inclusions $k\C \hookrightarrow d\C \hookrightarrow \C$. Consequently, we set $\Iso(k\C)$ and $\Iso(d\C)$ to be the semirings of isomorphism classes of K{\"u}nneth objects and dualizable objects, respectively, under $\oplus$ and $\otimes$. 

Lastly, we note that invertible objects are dualizable by \cite[Proposition 2.9]{May}, and an inverse of an invertible object serves as a canonical choice of dual. One can see, in general, that invertibility does not imply that an object is a K{\"u}nneth object. For example, $\Sigma^\infty S^1$ is an invertible spectrum which is not a retract of $\bigvee_{i=1}^n \sphere$ for any $n$. 

\subsection{An Algebraic Description of $\Pic(\Ho(\Sp^G))$}
Throughout, $\Ho(\Sp^G)$ will denote the $G$-equivariant stable homotopy category for $G$ a finite group. We regard $\Ho(\Sp^G)$ as a symmetric monoidal category equipped with its usual smash product and unit given by the sphere $G$-spectrum $\sphere_G \coloneqq \Sigma_G^\infty S^0$. In \cite{FLM}, Fausk, Lewis, and May give an exact sequence involving the Picard group of \(\Ho(\Sp^G)\). In fact, they prove the existence {of} a more general embedding theorem for Picard groups:

\begin{theorem}\cite[Theorem 0.2]{FLM}
Let $(\C, \otimes, \1)$ be a stable homotopy category. There is an embedding $\Pic(\End(\1)) \hookrightarrow \Pic(\C)$, where $\End(\1)$ is the ring of endomorphisms of $\1$. The objects in the image of this embedding are the invertible objects which are retracts of finite coproducts of copies of $\1$ (i.e. invertible K\"unneth objects). 
\end{theorem} 

Specializing to $\C = \Ho(\Sp^G)$, they go on to extend this embedding to an exact sequence:

\begin{theorem}\cite[Theorem 0.1]{FLM} \label{FLMSplitting}
Let $G$ be a compact Lie group. There is an exact sequence 
\[ 0 \rightarrow \Pic(A(G)) \rightarrow \Pic(\Ho(\Sp^G)) \rightarrow C(G), \]
where $C(G)$ is the additive group of integer-valued, continuous functions on the space of closed subgroups of $G$.
\end{theorem}

In the case that $G$ is a finite group, {in which the topology of $C(G)$ is discrete,} it immediately follows that the Picard group of the equivariant stable homotopy category decomposes as a direct sum $\Pic(\Ho(\Sp^G)) \cong \Z^r \oplus \Pic(A(G))$ for some finite $r$. The quotient $\Pic(\Ho(\Sp^G))/\Pic(A(G))$ is identified with $\Z^r$, where $r$ is the number of conjugacy classes of cyclic subgroups of $G$ {when $G$ is a nilpotent group (e.g. abelian).} See the beginning of \cite[Section 3]{Angeltveit} for a discussion of this fact which follows from \cite[Theorem III.5.4]{tD}. 

At this point, the obstruction to understanding $\Pic(\Ho(\Sp^G))$ abstractly is to understand $\Pic(A(G))$, and tom Dieck-Petrie continue to provide. 

\begin{theorem}\cite[(3.33)]{GeomMods}\cite[Theorem 5]{tDPic}
\label{PicBurnside}
For a finite group $G$, there is an isomorphism 
\[ \Pic(A(G)) \cong \prod\limits_{(H) \in \sfrac{\Sub(G)}{\text{conj}}} \left( \sfrac{\Z}{[{N_G(H)}:H]} \right)^\times / \{ \pm 1 \}, \]
where $\sfrac{\Sub(G)}{conj}$ is the set of conjugacy classes of subgroups of $G$ {and $N_G(H)$ is the normalizer of $H$ in $G$.} 
\end{theorem}

{Further discussion of \cref{PicBurnside} can be found in \cite[Section 6]{Krause}.} We obtain an additive description of $\Pic(\Ho(\Sp^G))$ in the case that $G$ is a finite group: 

\begin{theorem}\cite{FLM, GeomMods, tD}
For a finite group $G$, the Picard group of the $G$-equivariant stable homotopy category decomposes as a sum 
\[ \Pic(\Ho(\Sp^G)) \cong \Z^r \oplus \prod\limits_{(H) \in \sfrac{\Sub(G)}{\text{conj}}} \left( \sfrac{\Z}{[{N_G(H)}:H]} \right)^\times / \{ \pm 1 \}, \]
where $r$ is finite. 
\end{theorem}

\subsection{An Equivariant Algebraic Description of $\Pic(\Ho(\Sp^G))$}
\label{EquivAlgDescription}
The equivariant stable homotopy category is certainly an additive stable homotopy category. Consequently, the formalism of \cite{FLM} applies to $\Ho(\Sp^G)$ as we have seen. However, $\Ho(\Sp^G)$ can be equipped with a richer structure: a $\Mack_G$-enrichment. In order to recast the results of \cite{FLM} with a view towards this $\Mack_G$-enrichment, we first give a refined stratification of levels of dualizability in $G$-spectra. 

\begin{definition}
Let $X \in \Ho(\Sp^G)$ be a dualizable object. We say that $X$ is \textit{weakly K{\"u}nneth} if it is a retract of a finite wedge of induced spheres. That is, $X$ is a retract of $\bigvee_{i=1}^n (\sfrac{G}{H_i{}_+} \wedge \sphere_G)$ for some natural number \(n\). 
\end{definition}

It is immediately clear from the definition that K{\"u}nneth objects are weakly K{\"u}nneth objects: simply set $H_i = G$ for all $i$. We have the symmetric monoidal full subcategory $w\Ho(\Sp^G)$ of weakly K{\"u}nneth objects and strong symmetric monoidal inclusions 
\[ 
k\Ho(\Sp^G) \subseteq w\Ho(\Sp^G) \subseteq d\Ho(\Sp^G).
\]
Similarly, $\Iso(w\Ho(\Sp^G))$ denotes the semiring of weakly K{\"u}nneth objects under $\vee$ and $\wedge$. We begin our analysis with a result analogous to \cite[Proposition 1.2]{FLM}: 

\begin{proposition}\label{PropOfKunneth}
A dualizable object $X \in \Ho(\Sp^G)$ is weakly K{\"u}nneth if and only if any of the following equivalent conditions hold: 
\begin{enumerate}[(i)]
\item $\underline{\pi}_0(DX) \boxtimes \underline{\pi}_0(X) \rightarrow \underline{\pi}_0(DX \wedge X)$ is an isomorphism.
\item $\phi\colon \upi^0(Y) \boxtimes \upi_0(X) \cong \upi_0(DY) \boxtimes \upi_0(X) \rightarrow \upi_0(DY \wedge X) \cong \upi_0([Y,X])$ is an isomorphism for all dualizable objects $Y$. 
\item $\phi\colon \upi_0(Y) \boxtimes \upi_0(X) \rightarrow \upi_0(Y \wedge X)$ is an isomorphism for all dualizable objects $Y$.  
\end{enumerate}
\end{proposition}

\begin{proof}
See the proof of \cite[Proposition 1.2]{FLM}. {Note that a Mackey functor is finitely generated and projective if and only if it is a retract of the representable object $\underline{\pi}_0(\Sigma_G^\infty T_+ \wedge \sphere_G) = \B_G(-, T)$ for some finite $G$-set $T$.} 
\end{proof}

The above proposition generalizes \cite[Proposition 1.2]{FLM} in two ways: we take $\pi_0$ as valued in $\Mack_G$, but we also replace K\"unneth with weakly K\"unneth. This is because K\"unneth objects do not coincide with dualizable objects in $\Mack_G$
{, but what one would naturally call a weakly K\"unneth object in $\Mack_G$ is precisely a dualizable object.} In particular, if $X$ is a K\"unneth object in $\Ho(\Sp^G)$, then $\upi_0(X)$ is a K\"unneth object in $\Mack_G$. On the other hand, if $X$ is a weakly K\"unneth object in $\Ho(\Sp^G)$, then $\upi_0(X)$ is a dualizable or ``weakly K\"unneth" Mackey functor.


%
{
\begin{theorem}
\label{PiIso}
For $G$ a finite group, the map $\upi_0\colon \Iso(w\Ho(\Sp^G)) \rightarrow \Iso(dMack_G)$ is a isomorphism of semirings. In particular, taking units gives an isomorphism
\[ \upi_0\colon \Pic(w\Ho(\Sp^G)) \rightarrow \Pic(d\Mack_G) = \Pic(\Mack_G). \]
\end{theorem}

\begin{proof}
The argument that $\underline{\pi}_0$ is injective is analogous to that in \cite[Corollary 1.4]{FLM}. The assignment is surjective as we build a section $E_{(-)}$ in \cref{PicEmbeddings}.  
\end{proof}

We have analogous results for K{\"u}nneth objects. By definition, $\upi_0$ of a K{\"u}nneth object in $\Ho(\Sp^G)$ is a K{\"u}nneth object in $\Mack_G$. This gives a map $\upi_0\colon  \Iso(k\Ho(\Sp^G)) \rightarrow \Iso(kMack_G)$ of semirings.

\begin{theorem}
For $G$ a finite group, the map $\upi_0\colon  \Iso(k\Ho(\Sp^G)) \rightarrow \Iso(k\Mack_G)$ is an isomorphism of semirings. In particular, taking units gives an isomorphism
\[ \upi_0\colon \Pic(k\Ho(\Sp^G)) \rightarrow \Pic(k\Mack_G). \]
\end{theorem}

\begin{proof}
Again, injectivity is analogous to \cite[Corollary 1.4]{FLM} as every K\"unneth object is a weakly K\"unneth object, and surjectivity follows by virtue of having a section $kE_{(-)}$ built in \cref{PicEmbeddings}. 
\end{proof}
}

\begin{remark}\label{PotentialDegeneracy}
We have the following commutative square:
\[ 
\begin{tikzpicture}
\node (02) at (0,2) {$\Pic(k\Ho(\Sp^G))$};
\node (32) at (4,2) {$\Pic(w\Ho(\Sp^G))$};
\node (00) at (0,0) {$\Pic(k\Mack_G)$};
\node (30) at (4,0) {$\Pic(d\Mack_G)$};

\draw[right hook->] (02) to (32);
\draw[right hook->] (00) to (30);

\draw[->] (02) to node[left] {$\upi_0$} (00);
\draw[->] (32) to node[right] {$\upi_0$} (30);
\end{tikzpicture}
\]
It isn't remarkable that this diagram commutes; we simply emphasize that the bottom horizontal map {is not obviously an isomorphism} .  If one instead considers $A(G)$-modules, the following diagram commutes:  
\[ 
\begin{tikzpicture}
\node (02) at (0,2) {$\Pic(k\Ho(\Sp^G))$};
\node (32) at (4,2) {$\Pic(w\Ho(\Sp^G))$};
\node (00) at (0,0) {$\Pic(k\Mod_{A(G)})$};
\node (30) at (4,0) {$\Pic(d\Mod_{A(G)})$};

\draw[right hook->] (02) to (32);
\draw[->] (00) to node[below] {$\cong$} (30);

\draw[->] (02) to node[left] {$\pi_0$} (00);
\draw[->] (32) to node[right] {$\pi_0$} (30);
\end{tikzpicture}
\]
This comparison tells us that, a priori, considering $\Ho(\Sp^G)$ as $\Mack_G$-enriched gives a richer analysis of invertible $G$-spectra. However, we will see in \cref{BigTheorem} that the inclusion $\Pic(k\Mack_G) \hookrightarrow \Pic(d\Mack_G)$ is an isomorphism for finite groups. 
\end{remark}

\subsection{Building Invertible $G$-Spectra}
As the preceding theorems suggest, \cite{FLM} gives a construction for invertible K\"unneth objects in $\Ho(\Sp^G)$ coming from $\Pic(A(G))$. Moreover, this construction is inverse to the map induced by $\pi_0$. We give a somewhat different construction here whose input is $\Pic(\Mack_G)$.  

\begin{theorem}\label{PicEmbeddings}
Let $G$ be a finite group. There are embeddings
\[ kE_{(-)}\colon \Pic(k\Mack_G) \hookrightarrow \Pic(\Ho(\Sp^G)) \]
and 
\[ E_{(-)}\colon \Pic(\Mack_G) \hookrightarrow \Pic(\Ho(\Sp^G)) \]
whose images are the isomorphism classes of invertible K{\"u}nneth and weakly K{\"u}nneth objects in $G$-spectra, respectively. 
\end{theorem}

\begin{proof}
Throughout, we denote $\uparrow_{H_i}^G \underline{A}$ by $\underline{A}_{\sfrac{G}{H_i}}$. Let $\underline{M}$ be an invertible Mackey functor. Since invertible objects are dualizable, there exists a retraction 
\[ r\colon \bigoplus_i \underline{A}_{\sfrac{G}{H_i}} \twoheadrightarrow \underline{M} \]
with section $s\colon \underline{M} \hookrightarrow \bigoplus_i \underline{A}_{\sfrac{G}{H_i}}$. This yields an idempotent
\[ s \circ r \in \End\left( \bigoplus_i \underline{A}_{\sfrac{G}{H_i}} \right). \]
One can check that localizing $\bigoplus_i \underline{A}_{\sfrac{G}{H_i}}$ at this idempotent realizes $\underline{M}$. That is, 
\[ \underline{M} \cong \bigoplus_i \underline{A}_{\sfrac{G}{H_i}}[e_r^{-1}]. \]
Moreover, our idempotent lifts to an idempotent 
\[ e_r \in \End\left(\bigvee_i {\sfrac{G}{H_i}}_+ \wedge \sphere_G\right). \]
This follows from identifying $\upi_0(\vee_i {\sfrac{G}{H_i}}_+ \wedge \sphere_G)$ with $\oplus_i \underline{A}_{\sfrac{G}{H_i}}$. Localizing this wedge of spheres at $e_r$ yields a weakly K\"unneth (hence, connective) $G$-spectrum 
\[ E_{\underline{M}} \coloneqq \left(\bigvee_i {\sfrac{G}{H_i}}_+ \wedge \sphere_G\right)[e_r^{-1}]. \]

To see that $E_{\underline{M}}$ is invertible, we first observe that the homotopy of $E_{\underline{M}}$ is given by $\underline{M} \boxtimes \pi_*(\sphere_G)$ by construction. More generally, the homotopy of $E_{\underline{M}} \wedge X$ for any $X \in \Ho(\Sp^G)$ is given by 
\[ \underline{M} \boxtimes \upi_*(X). \]
One can leverage this description to show that $E_{\underline{M}}$ is independent of the choice of $\underline{M}$, retraction, and section. Indeed, if a different retraction or section is chosen, then one can build a map which induces a natural isomorphism of equivariant homology theories. An isomorphism of invertible Mackey functors will similarly yield a natural isomorphism of equivariant homology theories. In particular, we see that $E_{\underline{M}}$ is invertible with inverse given by $E_{\underline{M}^{-1}}$. 

We have constructed a set-theoretic section 
\[ E_{(-)}\colon \Pic(\Mack_G) \xrightarrow{\cong} \Pic(w\Ho(\Sp^G)) \]
of $\upi_0$, which is a group isomorphism by \cref{PiIso}. Consequently, $E_{(-)}$ is multiplicative by virtue of being the inverse of a group isomorphism. The analogous construction for invertible K\"unneth Mackey functors yields an embedding $kE_{(-)}\colon \Pic(k\Mack_G) \rightarrow \Pic(k\Ho(\Sp^G))$ which is inverse to $\upi_0\colon \Pic(k\Ho(\Sp^G)) \rightarrow \Pic(k\Mack_G)$.  
\end{proof}

\begin{corollary} \label{Folklore}
For $G$ a finite group, we have an isomorphism $\Pic(k\Mack(G)) \cong \Pic(A(G))$. 
\end{corollary}

\begin{proof}
Observe that $\Pic(k\Mack_G)$ and $\Pic(A(G))$ embed into $\Pic(\Ho(\Sp^G))$ with the same image. 
\end{proof}

Recall the \textit{$H$-geometric fixed points} functor $\Phi^H\colon \Ho(\Sp^G) \rightarrow \Ho(\Sp)$ for $(H)$ a conjugacy class of subgroup of $G$. {On objects, this functor is given by 
\[ \Phi^H(X) \coloneqq (\Sigma_H^\infty \widetilde{E\mathcal{F}}_H \wedge \downarrow_H^G X)^H, \]
where $\widetilde{E\mathcal{F}}_H$ is as described in the same way that $\widetilde{E\mathcal{F}}_N$ is described following \cref{GeomFixPtsFormula}. 
} 
Geometric fixed points is a strong symmetric monoidal functor which induces a map
\[ \Phi^H\colon \Pic(\Ho(\Sp^G)) \rightarrow \Pic(\Ho(\Sp)) \cong \Z. \]

\begin{definition}
Let $\Dim$ denote the map 
\[ \Pic(\Ho(\Sp^G)) \rightarrow C(G) \coloneqq \prod_{(H) \leq \Sub(G)/\text{conj}} \Z \]
induced by $\Dim(X)(H) = n$, where $n$ is such that $\Phi^H(X) \simeq \Sigma^n \sphere$. 
\end{definition}

Work of tom-Dieck and Petrie identifies the kernel of $\Dim$ as $\Pic(A(G))$, and the embedding $\Pic(A(G)) \rightarrow \Pic(\Ho(\Sp^G))$ given in \cite[Theorem 0.1]{FLM} presents $\Pic(A(G))$ as this kernel. 

{
\begin{theorem}
\label{Theorem A}
The image of $kE_{(-)}\colon \Pic(k\Mack_G) \rightarrow \Pic(\Ho(\Sp^G))$ is precisely the kernel of $\Dim$. That is, there is an exact sequence 
\[ 0 \rightarrow \Pic(k\Mack_G) \rightarrow \Pic(\Ho(\Sp^G)) \xrightarrow{\Dim} \Z^r \rightarrow 0 \]
which is split as $\Z^r$ is a free abelian group. 
\end{theorem} 
}

\begin{proof}
The image of $kE_{(-)}$ as above is equal to the image of $\Pic(A(G))$. 
\end{proof}

\begin{theorem}
\label{Theorem B}
The image of the embedding $E_{(-)}\colon \Pic(\Mack_G) \rightarrow \Pic(\Ho(\Sp^G))$ is contained in $\ker(\Dim)$. 
\end{theorem}

{
\begin{proof}
Let $\underline{M}$ be an invertible Mackey functor. First, $E_{\underline{M}}$ is connective as it is a retract of connective objects by construction. We have the following isomorphisms for $(H)$ a conjugacy class of subgroups of $G$: 
\begin{eqnarray*}
\pi_0(\Phi^H E_{\underline{M}}) &\cong & \upi_0((\downarrow_H^G E_{\underline{M}}) \wedge \widetilde{E\mathcal{F}}_H)(\sfrac{H}{H}) \\
&\cong & (\upi_0(\downarrow_H^G E_{\underline{M}}) \boxtimes \upi_0(\widetilde{E\mathcal{F}}_H))(\sfrac{H}{H}) \\
&\cong & (\downarrow_H^G \underline{M} \boxtimes \upi_0(\widetilde{E\mathcal{F}}_H) )(\sfrac{H}{H}) \\
&\cong & \qRes_e^H(\downarrow_H^G \underline{M} \boxtimes \underline{\pi}_0(\widetilde{E\mathcal{F}}_H)) \\
&\cong & \Phi^H (\downarrow_H^G \underline{M}) \\
&\cong & \Z.
\end{eqnarray*}
The second isomorphism follows from the observation that $\underline{\pi}_0(-)$ is a strong, symmetric monoidal functor on the full subcategory of connective $G$-spectra and the fifth isomorphism is as in \cref{GeomFixedPointsTop}. Since $\Phi^H E_{\underline{M}}$ is weakly equivalent to some suspension of $\sphere$, it must be that $\Phi^H E_{\underline{M}}$ is equivalent to $\sphere$. That is, $\Dim(E_{\underline{M}}) = 0$.  
\end{proof}
}

\begin{theorem} 
\label{BigTheorem}
Let $G$ be a finite group. The inclusion $k\Mack_G \subseteq \Mack_G$ induces an isomorphism on Picard groups $\Pic(k\Mack_G) \cong \Pic(\Mack_G)$. That is, every invertible Mackey functor is a K\"unneth object in $\Mack_G$.  
\end{theorem}

\begin{theorem}
\label{thm: PicA(G)isPicMackG}
For a finite group $G$, there is an isomorphism 
\[ \Pic(A(G)) \cong \Pic(\Mack_G). \]
\end{theorem}

\section{A Classification of Invertible Mackey Functors and $A(G)$-Modules}
\label{Classification Section}

In the case that $G=C_n$, a cyclic group of order $n$, there is already a classification of invertible Mackey functors given in \cite[Proposition 4.6]{Angeltveit}. It is shown that every invertible $C_n$-Mackey functor is isomorphic to some ``twisted Burnside Mackey functor". Here, we generalize this classification to all finite abelian groups with an approach heavily inspired by \cite{Angeltveit}. {As such, $G$ will denote a finite abelian group throughout this section.} 

We build explicit representatives for the classes in $\Pic(\Mack_G)$ and $\Pic(A(G))$. The $A(G)$-modules we construct will have underlying group given by $A(G)$ but come equipped with a module structure that is not induced by the ring structure on $A(G)$. For this reason, we call such a module a ``twisted Burnside module". Similarly, we will refer to the Mackey functors constructed here as ``twisted Burnside Mackey functors".  

\subsection{Twisted Burnside Mackey Functors}

The objects constructed here are obtained by scaling the restriction maps in the Burnside Mackey functor, and we refer to the scalings as ``twists". We start with the object that parameterizes these so-called ``twists". 

\begin{definition}
The \emph{Burnside twists of \(G\)}, denoted \(\tw{G}\), is defined by \( \tw{G} = \prod_{H \leq G} \Z \).
We will denote elements of $\tw{G}$ with lowercase Latin letters, such as \(a\), \(b\), \(c\), and so on. For any such \(a \in \tw{G}\) and subgroup \(H\) of \(G\), we write \(a_H\) for the image of \(a\) under the projection from \(T_G\) onto its \(H\)-factor. 

Here, $T_G$ is meant to be the product of commutative monoids under multiplication and as such inherits the structure of a commutative monoid via pointwise multiplication. 
\end{definition}

\begin{definition}
Given a Burnside twist \(a \in \tw{G}\) and a subgroup \(H\) of \(G\), the \emph{\(H\)-restriction factor} of \(a\), denoted \(r_H^a\), is the integer
\[
r_H^a = \prod_{K \geq H} a_K.
\]
{When no confusion will arise, we write $r_H^a$ as $r_H$. Note that \(r_G = a_G\) and that \(r_H\) divides \(r_J\) for any subgroup \(J\) of \(H\).} 
\end{definition}

\begin{definition}
Given a Burnside twist \(a \in \tw{G}\), the \emph{\(a\)-twisted Burnside Mackey functor} \(\uA^a\) is level-wise equal to \(\uA\) (i.e., \(\uA^a(\sfrac{G}{H}) = A(H)\) for every subgroup \(H\) of \(G\)), and has the same transfers and (trivial) Weyl actions as \(\uA\). However, the restrictions in \(\uA^a\) are scaled from those in \(\uA\) using the restriction factors of \(a\) as follows: for any generator \(\sfrac{H}{J}\) in \(A(H)\),
\[
\tres{a}{H}{K}(\sfrac{H}{J}) = \frac{r_{J \cap K}}{r_J} \cdot \res^H_K(\sfrac{H}{J}) = \frac{r_{J \cap K}}{r_J} \left| \sfrac{H}{JK} \right| \cdot \sfrac{K}{J \cap K}.
\]
We use $\tres{a}{H}{K}$ to denote the restriction from $H$ to $K$ in $\underline{A}^a$, except in the case of $a=1$, when $\underline{A}^a = \underline{A}$. In this case, we use $\res_H^K$ as above. 
\end{definition}

\begin{remark}
In the case that $a, b \in \tw{G}$ satisfies $a_H = b_H$ for all proper subgroups $H$ of $G$, $\underline{A}^a$ is \textit{equal} to $\underline{A}^b$. For this reason, we assume without loss of generality in all that follows that $a_G = 1$ for any $a \in \tw{G}$. 
\end{remark}

\begin{proposition}
The \(a\)-twisted Burnside Mackey functor \(\uA^a\) is indeed a Mackey functor for any twist \(a \in \tw{G}\).
\end{proposition}

\begin{proof}
First, note that transfers compose appropriately as they are the same as in $\underline{A}$. To see that the restrictions compose appropriately, let $J \leq K \leq H \leq G$ and $L \leq H$. Then, 
\begin{eqnarray*}
\tres{a}{K}{J} \left( \tres{a}{H}{K}(\sfrac{H}{L}) \right) &=& \tres{a}{K}{J} \left( \frac{r_{K \cap L}}{r_L} \cdot |\sfrac{H}{KL}| \cdot \sfrac{K}{K \cap L} \right) \\
&=& \frac{r_{K \cap L}}{r_L} \cdot |\sfrac{H}{KL}| \cdot \tres{a}{K}{J} \left (\sfrac{K}{K \cap L} \right) \\
&=& \frac{r_{K \cap L}}{r_L} \cdot |\sfrac{H}{KL}| \cdot \frac{r_{J \cap K \cap L}}{r_{K \cap L}} \cdot |\sfrac{K}{J(K \cap L)}| \cdot \sfrac{J}{J \cap K \cap L} \\
&=& \frac{r_{J \cap K \cap L}}{r_L} \cdot |\sfrac{H}{KL}| \cdot |\sfrac{K}{J(K \cap L)}| \cdot \sfrac{J}{J \cap K \cap L} \\
&=& \frac{r_{J \cap L}}{r_L} \cdot |H| \cdot \frac{|K \cap L|}{|K| \cdot |L|} \cdot |K| \cdot \frac{|J \cap K \cap L|}{|J| \cdot |K \cap L|} \cdot \sfrac{J}{J \cap L} \\
&=& 
\frac{r_{J \cap L}}{r_L}|H| \cdot \frac{|J \cap L|}{|J| \cdot |L|} \cdot \sfrac{J}{J \cap L} \\
&=& \frac{r_{J \cap L}}{r_L} \cdot |\sfrac{H}{JL}| \cdot \sfrac{J}{J \cap L} \\
&=& \tres{a}{H}{J}(\sfrac{H}{L}).
\end{eqnarray*}
The verification of the double coset formula is similarly straightforward. We conclude that $\underline{A}^a$ is a Mackey functor.  
\end{proof}

\begin{example}
Let $G = \mathcal{K} \coloneqq C_2 \times C_2$ be the Klein-four group. The subgroups of $\K$ consist of the trivial subgroup, three order two subgroups, and $\K$ itself. We will denote the order two subgroups as follows: 
\begin{itemize}
\item $L \coloneqq \langle (\gamma,e) \rangle$
\item $D \coloneqq \langle (\gamma,\gamma) \rangle$
\item $R \coloneqq \langle (e,\gamma) \rangle$,
\end{itemize}
where we denote the generator of $C_2$ by $\gamma$. We will write elements $\alpha \in \tw{G}$ as tuples $\alpha = (\alpha_e,\alpha_L,\alpha_D,\alpha_R, \alpha_\K)$. For $\alpha = (1,3,3,3,1)$, we have the following Lewis diagram for $\underline{A}^\alpha$ in which we have omitted all of the data which isn't changed by twisting:  
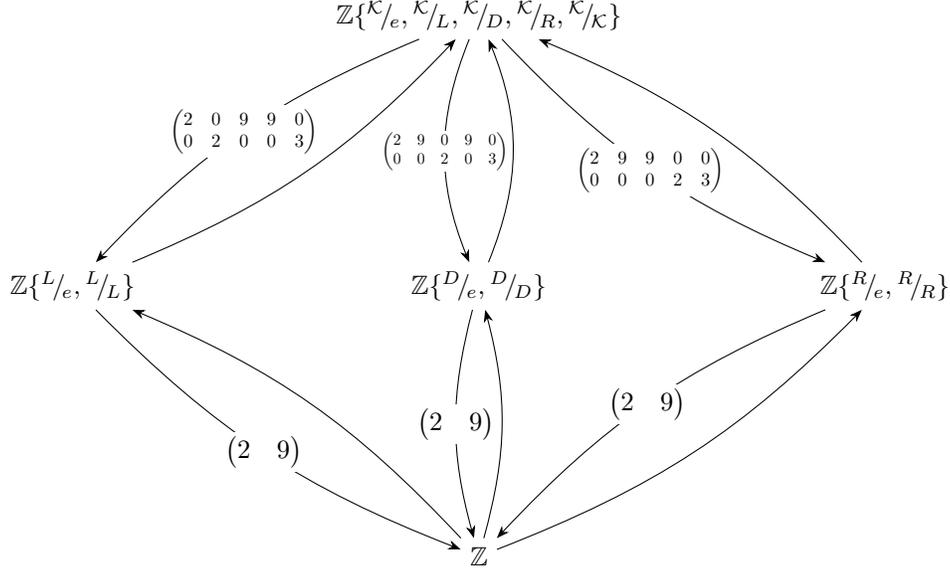
\[ 
\begin{tikzpicture}[scale=0.9]
		\node (K) at (0,4) {$\Z\{ \sfrac{\K}{e}, \sfrac{\K}{L}, \sfrac{\K}{D}, \sfrac{\K}{R}, \sfrac{\K}{\K}\}$};
		\node (L) at (-6,0) {$\Z\{ \sfrac{L}{e}, \sfrac{L}{L} \}$};
		\node (D) at (0,0) {$\Z\{ \sfrac{D}{e}, \sfrac{D}{D} \}$};
		\node (R) at (6,0) {$\Z\{ \sfrac{R}{e}, \sfrac{R}{R} \}$};
		\node (e) at (0,-4) {$\Z$};

		\draw[bend right=12,-Stealth] (e) to node[]{} (L);
		\draw[bend right=15,-Stealth] (e) to node[]{} (D);
		\draw[bend right=12,-Stealth] (e) to node[]{} (R);
		\draw[bend right=12,-Stealth] (L) to node[]{} (K);
		\draw[bend right=22,-Stealth] (D) to node[]{} (K);
		\draw[bend right=12,-Stealth] (R) to node[]{}(K);

		\draw[bend right=12,-Stealth] (L) to node[fill=white, inner sep=1.5pt]{$\mylittlematrix{2 & 9}$} (e);
		\draw[bend right=15,-Stealth] (D) to node[fill=white, inner sep=1.5pt]{$\mylittlematrix{2 & 9}$} (e);
		\draw[bend right=12,-Stealth] (R) to node[fill=white, inner sep=1.5pt]{$\mylittlematrix{2 & 9}$} (e);
		\draw[bend right=12,-Stealth] (K) to node[fill=white, inner sep=1.5pt]{$\mymidsizematrix{2 & 0 & 9 & 9 & 0 \\ 0 & 2 & 0 & 0 & 3}$} (L);
		\draw[bend right=22,-Stealth] (K) to node[fill=white, inner sep=0.5pt]{$\mytinymatrix{2 & 9 & 0 & 9 & 0 \\ 0 & 0 & 2 & 0 & 3}$} (D);
		\draw[bend right=12,-Stealth] (K) to node[fill=white, inner sep=1.5pt]{$\mymidsizematrix{2 & 9 & 9 & 0 & 0 \\ 0 & 0 & 0 & 2 & 3}$} (R);
	\end{tikzpicture}
\] 
\end{example}

\begin{lemma}
\label{lem: MapsOutOfTwisted}
For any Mackey functor $\underline{M}$ and Burnside twist $a \in \tw{G}$, the map
\[ \Hom(\underline{A}^a, \underline{M}) \rightarrow \bigoplus\limits_{H \leq G} \underline{M}(\sfrac{G}{H})^{W_G(H)} \]
given by $\varphi \mapsto ((\varphi_H(\sfrac{H}{H}))_{H \leq G}$ is injective with image 
\[ \Gamma^a(\underline{M}) \coloneqq \left\{ \gamma \in \bigoplus\limits_{H \leq G} \underline{M}(\sfrac{G}{H})^{W_G(H)} \mid \res_K^H(\gamma_H) = \frac{r^a_K}{r^a_H} \gamma_K \text{ for all $K \leq H \leq G$} \right\}.  \]
\end{lemma}

\begin{proof}
Let $\varphi, \psi \in \Hom(\underline{A}^a,\underline{M})$ and assume that $((\varphi_H(\sfrac{H}{H}))_{H \leq G}$ is equal to $((\psi_H(\sfrac{H}{H}))_{H \leq G}$. Since the elements $(\sfrac{H}{H})_{H \leq G}$ generate $\underline{A}^a$ under the transfers, an arbitrary map of Mackey functors $\underline{A}^a \rightarrow \underline{M}$ is uniquely determined by its values on $(\sfrac{H}{H})_{H \leq G}$. This shows that $\varphi$ agrees with $\psi$.  

It is easy to see that the image of the above construction is contained in $\Gamma^a(\underline{M})$. To see that this construction is surjective, let $(\gamma_H)_{H \leq G} \in \oplus_{H \leq G} \underline{M}(\sfrac{G}{H})^{W_G(H)}$ such that $\res(\gamma_G) = \frac{r^a_K}{r^a_H} \gamma_H$ for all $K \leq H \leq G$. The map $\varphi\colon \underline{A}^a \rightarrow \underline{M}$ determined by $\varphi(\sfrac{H}{H}) = \gamma_H$ is a map of Mackey functors since $\res(\gamma_G) = \frac{r^a_K}{r^a_H} \gamma_H$. Further, the image of $\varphi$ is as desired.
\end{proof}

\begin{proposition}
\label{ProdOfTwists}
Given $a,b \in T_G$, there is an isomorphism $\underline{A}^a \boxtimes \underline{A}^b \cong \underline{A}^{ab}$. 
\end{proposition}

{
\begin{proof}
By \cref{lem: MapsOutOfTwisted}, it suffices to show that the corepresentable functor $\Gamma^{ab}(-) \cong \Hom(\underline{A}^{ab}, -)$ is naturally isomorphic to $\Dress{\underline{A}^a}{\underline{A}^b}{-} \cong \Hom(\underline{A}^a \boxtimes \underline{A}^b, -)$, where $\Gamma^{ab}(-)$ is as in \cref{lem: MapsOutOfTwisted}. 

To build an isomorphism $\Phi_{\underline{M}}: \Hom(\underline{A}^{ab}, \underline{M}) \rightarrow \Dress{\underline{A}^a}{\underline{A}^b}{\underline{M}}$, natural in $\underline{M}$, recall that a map $\varphi: \underline{A}^{ab} \rightarrow \underline{M}$ determines elements $\{\gamma_H\}_{H \leq G} \in \Gamma^{ab}(\underline{M})$ by \cref{lem: MapsOutOfTwisted}. 
Consider the maps $\theta_H: \underline{A}^{a}(\sfrac{G}{H}) \otimes \underline{A}^{b}(\sfrac{G}{H})\rightarrow \underline{M}$ induced by the following values on basis elements: 
\[
\theta_H(\sfrac{H}{J} \otimes \sfrac{H}{K}) = \frac{r^a_{J \cap K} \cdot r^b_{J \cap K}}{r^a_J \cdot r^b_K} \tr_{J \cap K}^H(\gamma_{J \cap K}). 
\]
A direct computation shows that the maps $\{\theta_{H}\}_{H \leq G}$ uniquely assemble into the components of a Dress pairing $\Phi_{\underline{M}}(\varphi) \coloneqq \theta \in \Dress{\underline{A}^a}{\underline{A}^b}{\underline{M}}$. That is, $\Phi_{\underline{M}}$ is injective.

To see that $\Phi_{\underline{M}}$ is surjective, recall that any map $\psi' \in \Hom(\underline{A}^a \boxtimes \underline{A}^b, \underline{M})$ determines a Dress pairing $\theta' \in \Dress{\underline{A}^a}{\underline{A}^b}{\underline{M}}$. The properties outlined in \cref{MapsOutOfBoxProduct}, which are satisfied by components of Dress pairings, ensure that the components of $\theta'$ determine values $\{\gamma'_H\}_{H \leq G} \in \Gamma^{ab}(\underline{M})$. These values correspond to a map $\psi \in \Hom(\underline{A}^{ab}, \underline{M})$ such that $\Phi_{\underline{M}}(\psi) = \theta'$. Naturality of $\Phi$ in $\underline{M}$ follows from naturality of post-composition. By the Yoneda Lemma, we conclude that $\underline{A}^{ab}$ is isomorphic to $\underline{A}^a \boxtimes \underline{A}^b$. 
\end{proof}
}

Now, we investigate when two twisted Burnside Mackey functors are isomorphic in $\Mack_G$. Along the same lines as \cite[Proposition 4.6]{Angeltveit}, we have the following:

\begin{proposition}
The map $\underline{A}^{(-)}\colon \tw{G} \rightarrow (\{\text{iso. classes of Mackey functors} \}, \boxtimes, \underline{A})$ descends to a map 
\[ \prod_{H \leq G}\left( \sfrac{\Z}{[G:H]} \right)/ \{ \pm 1\} \rightarrow (\{\text{iso. classes of Mackey functors} \}, \boxtimes, \underline{A})
\] 
of commutative monoids.
\end{proposition}

\begin{proof}
To see that $\underline{A}^{(-)}$ descends to $\prod_{H \leq G} \sfrac{\Z}{[G:H]}$, we first observe that we can reduce to the case that $a, b \in \tw{G}$ satisfy $a_H = b_H$ for all subgroups aside from some fixed $F \leq G$ and $b_F = a_F + [G:F]$. We define 
\[
\varphi_H(\sfrac{H}{H}) = 
\begin{cases}
\sfrac{H}{H} - \frac{r_{F \cap H}^a}{a_F \cdot r_H^a} \cdot [G:FH] \cdot \sfrac{H}{F \cap H} & H \nleq F \\
\sfrac{H}{H} & H \leq F
\end{cases}
\]
and note that these values satisfy the conditions of \cref{lem: MapsOutOfTwisted} so as to determine a map $\varphi\colon \underline{A}^a \rightarrow \underline{A}^b$. Similarly, it is straightforward to verify that the inverse of $\varphi$ is the map $\psi\colon \underline{A}^b \rightarrow \underline{A}^a$ determined by 
\[
\psi_H(\sfrac{H}{H}) = 
\begin{cases}
\sfrac{H}{H} + \frac{r^b_{F \cap H}}{b_F \cdot r_H^b} \cdot [G:FH] \cdot \sfrac{H}{F \cap H} & H \nleq F \\
\sfrac{H}{H} & H \leq F
\end{cases}.
\]

To see that $\underline{A}^{(-)}$ descends further to $\prod_{H \leq G} (\sfrac{\Z}{[G:H]})/\{ \pm 1\}$, we can reduce to the case that $a,b \in \tw{G}$ such that $a_H = b_H$ for $H \neq F$ and $a_F = -b_F$. Again, \cref{lem: MapsOutOfTwisted} gives a map $\varphi\colon \underline{A}^a \rightarrow \underline{A}^b$ defined by $\varphi(\sfrac{H}{H}) = \sfrac{H}{H}$ for $H \nleq F$ and $\varphi(\sfrac{H}{H}) = -\sfrac{H}{H}$ for $H \leq F$. The inverse of this map is similarly defined.
\end{proof}

\begin{corollary}\label{TwistMapIntoPic}
{Let $G$ be a finite abelian group.} On invertible elements, we have a map 
\[ \underline{A}^{(-)}\colon \prod_{H \leq G} \left( \sfrac{\Z}{[G:H]} \right)^\times / \{ \pm 1 \} \rightarrow \Pic(\Mack_G). \] 
\end{corollary}

\begin{notation}
We set $\twinv{G} \coloneqq \prod_{H \leq G} \left( \sfrac{\Z}{[G:H]} \right)^\times / \{ \pm 1 \}$. 
\end{notation}

\subsection{A Classification of Invertible Mackey functors}

We first show that the map $\underline{A}^{(-)}$ as in \cref{TwistMapIntoPic} is an isomorphism onto the Picard group of K\"unneth Mackey functors, making use of the splitting given in \cref{PicSplitting}. One key observation is that inducing along a quotient map $G \twoheadrightarrow Q \coloneqq \sfrac{G}{N}$ preserves twisted Burnside Mackey functors. 

Throughout, we will fix a subgroup $N \leq G$ and $Q \coloneqq \sfrac{G}{N}$. We will also often use $J$ to denote $J/N$ when $J \geq N$, since there is a bijective correspondence between subgroups of $Q$ and subgroups of $G$ containing $N$. 

\begin{proposition}
\label{qIndOfTwisted}
Let $a \in \prod_{H \leq Q} \left( \sfrac{\Z}{[Q:H]} \right)^\times / \{ \pm 1 \} \cong \prod_{N \leq H \leq G} \left( \sfrac{\Z}{[G:H]} \right)^\times / \{ \pm 1 \}$. Then, $\qInd_Q^G \underline{A}^a$ is isomorphic to $\underline{A}^{\hat{a}} $, where 
\[
\hat{a}_{H} = 
\begin{cases}
a_H & H \geq N \\
1 & H \ngeq N
\end{cases}.
\]
\end{proposition}

\begin{proof}
For an arbitrary $\underline{M} \in \Mack_G$, we have the following chain of isomorphisms: 
\begin{eqnarray*}
\Hom(\qInd_Q^G \underline{A}^a, \underline{M}) &\cong& \Hom(\underline{A}^a, \qRes_Q^G \underline{M}) \\
&\cong&  \left\{ \gamma \in \bigoplus_{J \leq Q} (\qRes_Q^G \underline{M})(\sfrac{Q}{J})^{W_Q(J)} \mid \res_K^J \gamma_J = \frac{r^a_K}{r^a_J} \gamma_K \text{ for $K \leq J \leq Q$} \right\} \\
&\cong& \left\{\gamma \in \bigoplus_{N \leq J} \underline{M}(\sfrac{G}{J})^{W_G(J)} \mid \res_K^J \gamma_J = \frac{r^a_K}{r^a_J} \gamma_K \text{ for $N \leq K \leq J \leq {G}$} \right\} \\
&\cong& \left\{ \gamma \in \bigoplus_{J \leq G} \underline{M}(\sfrac{G}{J})^{W_G(J)} \mid \res_K^J \gamma_J = \frac{r^{\hat{a}}_K}{r^{\hat{a}}_J} \gamma_K \text{ for $K \leq J \leq G$} \right\} \\
& \cong& \Hom(\underline{A}^{\hat{a}},\underline{M}). 
\end{eqnarray*}
We conclude by the Yoneda Lemma that $\qInd_Q^G \underline{A}^a \cong \underline{A}^{\hat{a}}$.
\end{proof}

We can perform a similar analysis for $N$-geometric fixed points, using the adjunction between geometric fixed points and inflation. However, one could also see this through the explicit description of $N$-geometric fixed points. Levelwise, the geometric fixed points are given by quotienting by the image of appropriate transfers. Since twisted Burnside Mackey functors have the same levelwise data and transfers of the Burnside Mackey functor, it is easy to see that this will remain true after taking geometric fixed points. Further, the restrictions in geometric fixed points are induced by the original restrictions. 

\begin{proposition} 
Let $a \in \prod_{H \leq G} \left( \sfrac{\Z}{[G:H]} \right)^\times / \{ \pm 1 \}$. Then, $\Phi^N \underline{A}^a$ is isomorphic to $\underline{A}^\alpha$, where $\alpha_J = a_J$ for $J \leq Q$. 
\end{proposition}

\begin{proof}
One can leverage the aforementioned adjunction to build a natural isomorphism between $\Hom(\Phi^N \underline{A}^a, \underline{M})$ and $\Hom(\underline{A}^\alpha, \underline{M})$ for an arbitrary $\underline{M} \in \Mack_Q$. 
\end{proof}

\begin{corollary}
\label{MapOfSplits}
For any $N \leq G$ and $Q \coloneqq G/N$, there is a map of split exact sequences 
\begin{center}
\begin{tikzpicture}
\node (00) at (0,0) {$\Pic(\Mack_Q)$};
\node (50) at (5,0) {$\Pic(\Mack_G)$};
\node (100) at (10,0) {$\ker(\Phi^N)$};

\node(0-2) at (0,-2) {$\prod\limits_{J \geq N} \left( \sfrac{\Z}{[G:J]} \right)^\times / \{ \pm 1 \}$};
\node(5-2) at (5,-2) {$\prod\limits_{J \leq G} \left( \sfrac{\Z}{[G:J]} \right)^\times / \{ \pm 1 \}$};
\node(10-2) at (10,-2) {$\prod\limits_{J \ngeq N} \left( \sfrac{\Z}{[G:J]} \right)^\times / \{ \pm 1 \}$};

\draw[right hook->] (00) to node[above] {$\qInd_Q^G$} (50);
\draw[->>] (50) to (100);

\draw[right hook->] (0-2) to (5-2);
\draw[->>] (5-2) to (10-2);

\draw[->] (0-2) to node[left] {$\underline{A}^{(-)}$} (00);
\draw[->] (5-2) to node[fill=white] {$\underline{A}^{(-)}$} (50);
\draw[->] (10-2) to node[right] {$\underline{A}^{(-)}$} (100);
\end{tikzpicture}.
\end{center}
\end{corollary}

Before we proceed with the classification, we record a useful lemma concerning the image of $\underline{A}^{(-)}$ in the kernel of geometric fixed points. 

\begin{lemma}
\label{MapsOutOfKerPhi}
Let $\overline{a} \in \prod_{J \ngeq N} \left( \sfrac{\Z}{[G:J]} \right)^\times / \{ \pm 1 \}$ and $a \in \twinv{G}$ obtained by extending $\overline{a}$ by 1's. For any $\underline{P} \in \Mack_G$, we have an isomorphism 
\[ \Hom(\underline{A}^a, \underline{P}) \cong \left\{ \gamma \in \underline{P}(\sfrac{G}{G}) \oplus \bigoplus\limits_{J \ngeq N} \underline{P}(\sfrac{G}{J})^{W_G(J)} \mid \res_J^K \gamma_K = \frac{r_J^{a}}{r_K^{a}} \gamma_J \text{ for $J \leq K$} \right\}
\]   
\end{lemma}

\begin{proof}
Combining \cref{lem: MapsOutOfTwisted} with the observation that we have $\res_H^G \gamma_G = r_H^{\overline{a}} \gamma_H = \gamma_H$ for $H \geq N$ yields the isomorphism. 
\end{proof}

\begin{theorem}
\label{Theorem C}
For $G$ a finite abelian group, the map $\underline{A}^{(-)}\colon \prod_{J \leq G} \left( \sfrac{\Z}{[G:J]} \right)^\times / \{ \pm 1 \} \rightarrow \Pic(\Mack_G)$ is injective. 
\end{theorem}

\begin{proof}
Let \(a \in \prod_{H \leq G} \left( \sfrac{\Z}{[G:H]} \right)^\times / \{ \pm 1 \}\) and suppose \(\phi\colon \uA^a \to \uA\) is an isomorphism. If \(G = e\), the theorem holds trivially. So, we suppose \(|G| = n > 1\) and induct on the number of prime factors of \(n\) (counted with multiplicity). For the base case, \(n = p\) is prime, so \(G = C_p\). The only possibly non-trivial component of the twist \(a\) is \(a_e\) in this case. However, invertibility of \(\phi\) and commutativity with transfers implies that \(\phi_{C_p}(\sfrac{C_p}{C_p}) = n \cdot \sfrac{C_p}{e} \pm \sfrac{C_p}{C_p}\) for some integer \(n\). Commutativity with restrictions then implies \(\pm a_e = np \pm 1\), so \(a_e \equiv \pm 1 \pmod{[C_p: e]}\). Thus, \(\uA^{(-)}\) is injective in this case.

If \(n\) has more than one prime factor, let \(N \leq G\) be any non-trivial proper subgroup and define
\begin{align*}
\hat{a}_J &= 
\begin{cases}
a_J & J \geq N \\
1 & \text{ else}
\end{cases} \\
\overline{a}_J &= 
\begin{cases}
a_J & J \ngeq N \\
1 & \text{ else}
\end{cases}.
\end{align*}
It follows that $\underline{A}^a$ decomposes as $\underline{A}^{\hat{a}} \boxtimes \underline{A}^{\overline{a}}$. Further, \cref{qIndOfTwisted} shows that $\underline{A}^{\hat{a}}$ is $\qInd_Q^G(\underline{A}^{\alpha})$, where ${\alpha}_{\sfrac{J}{N}} = a_J$ for $J \geq N$. It is also straightforward to check that $\Phi^N \underline{A}^{\overline{a}}$ is isomorphic to $\underline{A}$. Now, \cref{PicSplitting} and injectivity of $\qInd_Q^G$ imply that both $\underline{A}^{\alpha}$ and $\underline{A}^{\overline{a}}$ are trivial. Since \(|N|\) and \(|Q|\) are proper divisors of \(|G|\), they have fewer prime factors, so our induction hypothesis yields
\[ a_N = {\alpha}_{\sfrac{N}{N}} \equiv \pm 1 \pmod{[Q:e]}. \]
As \([Q : e]\) agrees with \([G : N]\), we conclude that $\underline{A}^{(-)}$ is injective. 
\end{proof}

\begin{corollary}\label{ClassificationThm}
For $G$ a finite abelian group, the embedding
\[ \underline{A}^{(-)}\colon \twinv{G} \rightarrow \Pic(\Mack_G) \]
is an isomorphism. That is, every invertible Mackey functor is isomorphic to some twisted Burnside Mackey functor. 
\end{corollary}

\begin{proof}
The cardinality of $\twinv{G}$ is finite and equal to that of $\Pic(A(G)) \cong \Pic(\Mack_G)$, and the assignment is injective by \cref{Theorem C}.   
\end{proof}

\subsection{Twisted Burnside Modules} 

Using the twisted Burnside Mackey functors we have constructed, we can produce explicit representatives of the classes in $\Pic(A(G))$ via an adjunction relating modules over the Burnside ring to Mackey functors. 

\begin{definition}
Given an $A(G)$-module $M$, we can build a Mackey functor $\underline{A} \otimes_{A(G)} M$ which is defined at level $\sfrac{G}{H}$ to be $A(H) \otimes _{A(G)} M$. The restrictions and transfers in $\underline{A} \otimes_{A(G)} M$ are given by tensoring the identity on $M$ with the restrictions and transfers in $\underline{A}$, respectively. 
\end{definition}

This tensoring construction has a right adjoint which is given by evaluating Mackey functors at $\sfrac{G}{G}$. 

\begin{definition}
There is an evaluation functor $\ev_{\sfrac{G}{G}}\colon \Mack_G \rightarrow \Mod_{A(G)}$ given on objects by $\ev_{\sfrac{G}{G}}(\underline{M}) \coloneqq \underline{M}(\sfrac{G}{G})$. We view $\ev_{\sfrac{G}{G}}\underline{M}$ as an $A(G)$-module via 
\[ \sfrac{G}{H} \cdot m = \tr_H^G \res_H^G m \]
for any $m \in \underline{M}(\sfrac{G}{G})$. 
\end{definition}

\begin{proposition}\label{EvalTensorAdjunction}
Let $G$ be a finite group. There is an adjunction 
\[ \adjunction{$\Mod_{A(G)}$}{$\Mack_G$}{$\underline{A} \otimes_{A(G)} (-)$}{$\ev_{\sfrac{G}{G}}$}. \]
Further, the composite $\ev_{\sfrac{G}{G}}(\underline{A} \otimes_{A(G)} (-))$ is naturally isomorphic to the identity. 
\end{proposition}

{
\begin{proof}
Given a map $\varphi \in \Hom(\underline{A} \otimes_{A(G)} M, \underline{N})$, the map $\varphi_{\sfrac{G}{G}}: M \rightarrow \underline{N}(\sfrac{G}{G})$ is a map of $A(G)$-modules. This is because the level-wise maps in maps of Mackey functors commute independently with restrictions/transfers and these operations determine the $A(G)$-module structure on $\underline{N}(\sfrac{G}{G})$. Conversely, suppose that $\psi: M \rightarrow \underline{N}(\sfrac{G}{G})$ is a map of $A(G)$-modules. The maps $\psi(\sfrac{H}{H}): A(H) \otimes_{A(G)} M \rightarrow \underline{N}(\sfrac{G}{H})$ determined by 
\[ \sfrac{H}{J} \otimes m = \tr_J^H \res_J^G \sfrac{G}{G} \otimes m \mapsto \tr_J^H \res_J^G \psi(m) \]
assemble into a map of Mackey functors $\underline{A} \otimes_{A(G)} M \rightarrow \underline{N}$. It is straightforward to verify that these constructions are inverse to one another and that this bijection is natural. 

By direct computation, $\ev_{\sfrac{G}{G}}(\underline{A} \otimes_{A(G)} M) = A(G) \otimes_{A(G)} M$ and this is naturally isomorphic to $M$. 
\end{proof}
}

\begin{proposition}
The functor $\underline{A} \otimes_{A(G)} (-)\colon \Mod_{A(G)} \rightarrow \Mack_G$ is strong symmetric monoidal. 
\end{proposition}

\begin{proof}
We proceed by the Yoneda Lemma. Let $M,N \in \Mod_{A(G)}$ and $\underline{P}$ be an arbitrary Mackey functor. We have 
\begin{eqnarray*}
\Hom(\underline{A} \otimes_{A(G)} (M \otimes_{A(G)} N),\underline{P}) &\cong& \Hom(M \otimes_{A(G)} N, \underline{P}(\sfrac{G}{G})) \\
&\cong& \Hom(M, \Hom(N, \underline{P}(\sfrac{G}{G}))) \\
&\cong& \Hom(M, \underline{\Hom}(\underline{A} \otimes_{A(G)} N, \underline{P})(\sfrac{G}{G})) \\
&\cong & \Hom(\underline{A} \otimes_{A(G)} M, \underline{\Hom}(\underline{A} \otimes_{A(G)} N, \underline{P})) \\
&\cong & \Hom((\underline{A} \otimes_{A(G)} M) \boxtimes (\underline{A} \otimes_{A(G)} N), \underline{P}).
\end{eqnarray*} 
\end{proof}

\begin{definition}
For $a \in \tw{G}$, the \textit{$a$-twisted Burnside module}, denoted $A(G)^a$, is the (left) $A(G)$-module given by $\ev_{\sfrac{G}{G}} \underline{A}^a$. Unpacking the module structure gives an explicit formula for the action of $A(G)$ on $A(G)^a$:
\[
\sfrac{G}{H} \cdot \sfrac{G}{J} = \tr_H^G\tres{a}{G}{H} \sfrac{G}{J} = \frac{r_{J \cap K}}{r_H} \cdot [G:JK] \cdot \sfrac{J}{J \cap K}. 
\]   
\end{definition}

\begin{warning}
\label{Warning}
The functor $\ev_{\sfrac{G}{G}}\colon \Mack_G \rightarrow \Mod_{A(G)}$ is \textit{not} a strong symmetric monoidal functor. This fails already when $G=C_2$: let $F\Z_{\sigma}$ denote the fixed points Mackey functor associated to the $C_2$-module $\Z_\sigma$. Then, $\ev_{\sfrac{G}{G}} F\Z_\sigma = 0$ and one can check via Lewis' description of the box product that $F\Z_{\sigma} \boxtimes F\Z_{\sigma}$ is isomorphic to $\underline{\Z}^*$, the dual constant Mackey functor at $\Z$. Notably, the evaluation of $\Z^*$ at $\sfrac{G}{G}$ is non-zero. 
\end{warning}

The failure of $\ev_{\sfrac{G}{G}}$ to be strong symmetric monoidal means that we cannot deduce that $\ev_{\sfrac{G}{G}}$ descends to a map on Picard groups. In particular, we cannot conclude immediately that twisted Burnside modules are invertible. It turns out that restricting our attention to K\"unneth Mackey functors greatly remedies this problem, and the first step in this restricted analysis is an instantiation of another result of \cite{FLM}.

\begin{prop}\label{EvalOnKunneth}
The functor $\ev_{\sfrac{G}{G}}\colon \Mack_G \rightarrow \Mod_{A(G)}$ restricts to a \textit{strong symmetric monoidal functor} 
\[ \ev_{\sfrac{G}{G}}\colon k\Mack_G \rightarrow k\Mod_{A(G)}. \] 
\end{prop}

\begin{proof}
In \cite[Proposition 1.2(iii)]{FLM}, consider K\"unneth objects in $\Mack_G$ and note that $\pi_0$ translates to $\ev_{\sfrac{G}{G}}$.  
\end{proof}

Similarly, it is easy to verify the following: 

\begin{prop}
The functor $\underline{A} \otimes_{A(G)} (-)$ preserves K\"unneth objects, hence it restricts to a strong symmetric monoidal functor 
\[ k\Mod_{A(G)} \rightarrow k\Mack_G. \] 
\end{prop}

\begin{proof}
Since $\underline{A} \otimes_{A(G)} (-)$ is an additive, strong symmetric monoidal functor, we have 
\[ \underline{A} \otimes_{A(G)} \left( \bigoplus_i A(G) \right) \cong \bigoplus_i \underline{A}. \]
Further, $\underline{A} \otimes (-)$ preserves retractions.
\end{proof}
 
\begin{theorem}
\label{EvalEquiv}
The functor $\ev_{\sfrac{G}{G}}\colon k\Mack_G \rightarrow k\Mod_{A(G)}$ is a symmetric monoidal equivalence of categories. Consequently, it gives an isomorphism of Picard groups
\[ \ev_{\sfrac{G}{G}}\colon \Pic(k\Mack_G) \rightarrow \Pic(k\Mod_{A(G)}) \cong \Pic(A(G)). \] 
\end{theorem}

\begin{proof}
To see that $\ev_{\sfrac{G}{G}}$ is fully faithful, we observe that K\"unneth objects in $\Mack_G$ are finitely generated as Mackey functors at $\underline{M}(\sfrac{G}{G})$. This implies that maps out of a K\"unneth Mackey functor are uniquely determined by their evaluation at $\sfrac{G}{G}$. On the other hand, $\ev_{\sfrac{G}{G}}$ is essentially surjective by virtue of having a right inverse as in \cref{EvalTensorAdjunction}. 
\end{proof}

As twisted $A(G)$-modules are defined precisely as the $(\sfrac{G}{G})$-evaluation of twisted Burnside Mackey functors, \cref{ClassificationThm} furnishes a classification of invertible $A(G)$-modules through the previous results. 

\begin{corollary}
\label{maincorollary}
Let $G$ be a finite abelian group. The map
\[ \prod_{H \leq G} \left( \sfrac{\Z}{[G:H]} \right)^\times / \{\pm 1 \} \xrightarrow{\underline{A}^{(-)}} \Pic(k\Mack_G) \xrightarrow{\ev_{\sfrac{G}{G}}} \Pic(\Mod_{A(G)}) \]
is an isomorphism with composite given by 
\[ A(G)^{(-)}\colon \prod_{H \leq G} \left( \sfrac{\Z}{[G:H]} \right)^\times / \{\pm 1 \} \rightarrow \Pic(\Mod_{A(G)}). \]
That is, every invertible $A(G)$-module is isomorphic to a twisted Burnside module. 
\end{corollary}

\begin{proof}
$\underline{A}^{(-)}$ is an isomorphism by \cref{ClassificationThm} and $\ev_{\sfrac{G}{G}}$ induces an isomorphism on Picard groups by \cref{EvalEquiv}. The composite is as stated by definition. 
\end{proof}

\clearpage

\bibliographystyle{amsalpha}

\end{document}